\documentclass[reqno]{amsart}

\usepackage[utf8]{inputenc}
\usepackage{amsmath}
\numberwithin{equation}{section}
\numberwithin{figure}{section}
\numberwithin{table}{section}
\usepackage{graphicx}
\usepackage{amssymb}
\usepackage{mathabx}
\usepackage{mathrsfs}
\usepackage{amsfonts}
\usepackage[colorlinks=true, allcolors=blue]{hyperref}
\usepackage{subfigure} 
\usepackage{color}
\usepackage{extarrows}
\usepackage{booktabs}
\usepackage{hyperref}
\usepackage{tikz}
\usepackage{extarrows}
\usepackage{booktabs}
\usepackage{amsthm}
\usepackage{enumerate}
\usepackage{enumitem} 
\usepackage{indentfirst}
\usepackage{embrac}
\usepackage{cite} 
\usetikzlibrary{positioning,shapes}

\def\dd{\mathrm{d}}

\def\D{\mathcal{D}}

\def\bR{{\mathbb R}}

\def\FF{{\mathscr F}}

\def\sG{{\mathscr G}}

\def\i{{\mathrm{i}}}
\def\L{{\mathcal{L}}}

\def\sL{\mathscr{L}}
\def\cD{\mathcal{D}}

\def\bN{\mathbb{N}}

\def\bP{\mathbb{P}}
\def\fp{\mathfrak{p}}
\def\sD{\mathscr{D}}
\def\bQ{\mathbb{Q}}
\def\cB{\mathcal{B}}
\def\sQ{\mathscr{Q}}
\def\fc{\mathfrak{c}}

\def\limsup{\mathop{\overline{\mathrm{lim}}}}
\def\liminf{\mathop{\underline{\mathrm{lim}}}}

\def\${|\!|\!|}
\def\l|{\left|\!\left|\!\left|}
\def\r|{\right|\!\right|\!\right|}

\def\bE{\mathbb{E}}

\newtheorem{theorem}{Theorem}[section]
\newtheorem{lemma}[theorem]{Lemma}
\newtheorem{proposition}[theorem]{Proposition}
\newtheorem{corollary}[theorem]{Corollary}
\theoremstyle{definition}
\newtheorem{definition}[theorem]{Definition}

\theoremstyle{remark}
\newtheorem{remark}[theorem]{Remark}
\numberwithin{equation}{section}

\author{ Liping Li}
\address{Fudan University, Shanghai, China.}
\email{liliping@fudan.edu.cn}
\thanks{The author is a member of LMNS,  Fudan University.  He is partially supported by NSFC (No. 12371144).}

\author{Zhangjie Wang}
\address{Fudan University, Shanghai, China.}
\email{zjwang25@m.fudan.edu,cn}
\begin{document}

\title[From BRWs to FBMs]{From boundary random walks to Feller's Brownian Motions}

%    General info



\begin{abstract}
We establish an invariance principle connecting boundary random walks on $\mathbb{N}$ with Feller’s Brownian motions on $[0,\infty)$. A Feller’s Brownian motion is a Feller process on $[0,\infty)$ whose excursions away from the boundary $0$ coincide with those of a killed Brownian motion, while its behavior at the boundary is characterized by a quadruple $(p_1,p_2,p_3,p_4)$. This class encompasses many classical models, including absorbed, reflected, elastic, and sticky Brownian motions, and further allows boundary jumps from $0$ governed by the measure $p_4$. For any Feller’s Brownian motion that is not purely driven by jumps at the boundary, we construct a sequence of boundary random walks whose appropriately rescaled processes converge weakly to the given Feller’s Brownian motion. %This framework not only captures the classical boundary behaviors, but also accommodates boundary jumps specified by the measure $p_4$.

\end{abstract}

\maketitle
\tableofcontents

\section{Introduction}\label{SEC1}

The approximation of stochastic processes on continuous state spaces by discrete models constitutes a classical theme in probability theory, motivated by both practical exigencies and theoretical considerations. From a practical standpoint, discrete models are indispensable for numerical simulations, computational implementations, and Monte Carlo–type methods, in which continuous dynamics must be represented by finite or countable structures. From a theoretical perspective, invariance principles provide a rigorous mathematical framework that explains how macroscopic continuous behavior emerges from microscopic discrete dynamics under appropriate scaling limits. Since the seminal contributions of Donsker \cite{D51, D52}, who established the functional central limit theorem for random walks, such invariance principles have played a fundamental role in justifying the approximation of Brownian motion by discrete counterparts. These foundational results underpin a vast body of work on diffusion limits and continue to serve as a guiding paradigm for the approximation of continuous stochastic dynamics by discrete processes. For comprehensive and systematic treatments, we refer the reader to Billingsley \cite{B99} and Ethier and Kurtz \cite{EK86}.

The approximation of stochastic processes becomes substantially more intricate when the state space is constrained by a boundary. Even in one dimension, boundary effects can fundamentally alter the qualitative behavior of the limiting process. Classical examples illustrating this complexity include absorbed, reflected, and sticky Brownian motions on the half-line $[0,\infty)$, whose construction and properties have been studied extensively since the early development of diffusion theory (see, e.g., Feller \cite{F54}; It\^o and McKean \cite{IM65}).
At the discrete level, absorbed Brownian motion is naturally approximated by random walks with absorbing at the boundary, while reflected Brownian motion can be obtained through suitable reflection mechanisms. Approximation schemes for reflected Brownian motion have been investigated in a variety of settings, including multidimensional domains; see, e.g., Burdzy et al. \cite{BC08}. Sticky Brownian motion, which spends a positive amount of time at the boundary, arises from more delicate boundary interactions, and its invariance principles have also been established for appropriate classes of random walks; see, for instance, Harrison et al. \cite{HL81} and Amir \cite{A91}. These results indicate that invariance principles in the presence of boundaries require more delicate analytical arguments. Indeed, the limiting dynamics depend not only on the local behavior of the process in the interior, but also crucially on the asymptotic manner in which the discrete process interacts with the boundary.

%The approximation of stochastic processes becomes substantially more intricate when the state space is constrained by a boundary. Even in one dimension, boundary effects can fundamentally alter the qualitative behavior of the limiting process. Classical examples illustrating this complexity are absorbed and reflected Brownian motions on the half-line $[0,\infty)$, whose construction and properties have been studied extensively since the early development of diffusion theory (see, e.g., Feller \cite{F54}; It\^o and McKean \cite{IM65}).At the discrete level, absorbed Brownian motion is naturally approximated by random walks with absorbing at the boundary. For reflected Brownian motion, approximation schemes have been investigated in a variety of settings, including multidimensional domains; see, e.g., Burdzy et al.  \cite{BC08}. These results indicate that invariance principles in the presence of boundaries require more delicate analytical arguments. Indeed, the limiting dynamics depend not only on the local behavior of the process in the interior, but also crucially on the asymptotic manner in which the discrete process interacts with the boundary.

A particularly natural and flexible framework for describing one-dimensional diffusions with boundary effects is provided by Feller's Brownian motions on $[0,\infty)$. In his pioneering work, Feller \cite{F52} gave a complete characterization of the admissible boundary behaviors at $0$ for Brownian motion on the half-line in terms of boundary conditions imposed on the infinitesimal generator. More precisely, the generator acts as $\frac{1}{2}f''$, while its domain consists of functions satisfying the boundary condition
\begin{equation}\label{eq:11}
p_1 f(0)-p_2 f'(0)+\frac{p_3}{2} f''(0)
+\int_{(0,\infty)}\bigl(f(0)-f(x)\bigr)\,p_4(\mathrm{d}x)=0,
\end{equation}
where $p_1,p_2,p_3$ are nonnegative constants and $p_4$ is a measure on $(0,\infty)$ satisfying $\int_{(0,\infty)}(1\wedge x)\,p_4(\mathrm{d}x)<\infty$. This formulation unifies the classical boundary conditions of absorption, reflection, elasticity, and stickiness, and moreover allows for jump-in behavior from the boundary governed by the measure $p_4$. The resulting processes display a rich and sometimes intricate boundary behavior, which was analyzed in detail by It\^o and McKean \cite{IM63}. For a more intuitive discussion of the boundary dynamics associated with Feller's Brownian motions, we refer to Section~\ref{SEC22} below.

In recent years, invariance principles for discrete processes approximating Brownian motions with boundary behavior have attracted renewed attention, particularly in models where different boundary mechanisms are combined within a single framework. In this direction, Erhard et al. \cite{EFJP24} and Arroyove et al. \cite{ABP25} study invariance principles for random walks converging to Brownian motions with various combinations of continuous boundary behaviors, and also provide discussions and references to further recent work in this area. Their approach is well suited to absorbed, reflected, elastic and sticky boundary conditions; however, boundary dynamics involving jumps are explicitly excluded from their setting.

The main goal of this paper is to provide a discrete approximation scheme for general Feller’s Brownian motions on $[0,\infty)$ using \emph{boundary random walks}, that is, (simple) random walks on $\mathbb{N}$ which, upon reaching the boundary $0$, may jump to various positions in $\mathbb{N}$. We show that a broad class of Feller's Brownian motions, whose boundary dynamics are not purely jump-driven, arise as weak limits of appropriately rescaled boundary random walks. In particular, the limiting Feller's Brownian motions may exhibit boundary jumps governed by the measure $p_4$, including the case where $p_4$ is infinite, while our framework also recovers the classical cases without boundary jumps. The precise statements of these results are presented in Theorems~\ref{THM31} and \ref{THM35}. To our knowledge, discrete approximation schemes leading to such general boundary jump mechanisms in the limit have not been previously established. 

Our approach employs a conceptually transparent two-step strategy. First, tightness of the rescaled processes is established via the classical Aldous criterion (see \cite{Al78}), which is particularly well suited for sequences of c\`adl\`ag processes that may exhibit jumps. Second, the identification of the limiting process follows the strategy of Arroyove et al. \cite{ABP25}, exploiting the uniqueness of the martingale problem associated with the target Feller process. This approach avoids the need for explicit generator convergence and builds on recent ideas developed for diffusion limits with continuous boundary behavior, extending them to accommodate boundary jumps.

Our results do not cover the extreme case in which the boundary behavior is purely jump-driven, that is, when $p_2 = p_3 = 0$ and the jumping measure $p_4$ is infinite (see Section~\ref{SEC33} for further discussion). In this case, the Feller's Brownian motion lacks any continuous mechanism for entering the interior from the boundary, and the techniques employed here, particularly those used in the tightness analysis, are no longer applicable. Addressing purely jump-driven boundary behavior appears to require fundamentally different ideas, and we leave this problem for future investigation.

The paper is organized as follows. In Section~\ref{SEC2}, we recall the general definition and characterization of Feller's Brownian motions on the half-line and introduce the class of boundary random walks employed for their approximation. Section~\ref{SEC3} provides a heuristic discussion of the scaling scheme, along with an informal presentation of the main results and the key ideas underlying their proofs. The main invariance principles are stated rigorously in Section~\ref{SEC4}, and the remaining sections are devoted to the detailed proofs of these results.

\section{Boundary Random Walk and Feller's Brownian Motion}\label{SEC2}

The primary purpose of this section is to introduce two types of Markov processes: a discrete-time Markov chain and a continuous-time Feller process. The common feature of these processes is that they exhibit analogous behaviors at the respective boundaries of their state spaces.

\subsection{Boundary Random Walk}
Consider a discrete-time Markov chain
\[
X=(X_k)_{k\geq 0}
\]
defined on a probability space $(\Omega,\mathscr{G},\bP)$ with the state space $\bN=\{0,1,2,\cdots\}$. Its transition matrix is given by
\begin{equation}\label{eq:22}
P=(p_{ij})_{i,j\in \bN}=\begin{pmatrix}
\mathfrak{p}_0 & \mathfrak{p}_1 & \mathfrak{p}_2 & \mathfrak{p}_3 & \cdots \\
\frac{1}{2} & 0 & \frac{1}{2} & 0 & \cdots \\
0 & \frac{1}{2} & 0 & \frac{1}{2} & \cdots \\
0 & 0 & \frac{1}{2} & 0 & \cdots \\
\cdots & \cdots & \cdots & \cdots & \cdots
\end{pmatrix},
\end{equation}
where $\mathfrak{p}_j$, $j\in \bN$, are non-negative constants such that
\begin{equation}\label{eq:21}
    \sum_{j\in \bN}\mathfrak{p}_j\leq 1.
\end{equation}
More precisely, for any $i\geq 1$, one has $p_{ij}=\frac{1}{2}$ if $j=i-1$ or $j=i+1$, and $p_{ij}=0$ otherwise. Additionally, $p_{0j}=\mathfrak{p}_j$ for all $j\in \bN$. Let $\bN_\Delta:=\bN\cup \{\Delta\}$ denote the one-point compactification of $\bN$ with respect to the relative topology induced by the extended real number system, where $\Delta$ is interpreted as the cemetery state of $X$. Note that \eqref{eq:21} implies
\[
\mathfrak{p}_\Delta:=1-\sum_{j\in \bN}\mathfrak{p}_j\geq 0.
\]
With this convention, from any state $i\ge 1$ the Markov chain $X$ moves to its two neighboring states with equal probability. Once it reaches the boundary state $0$, however, it may jump to any $j$ with probability $\mathfrak{p}_j$, or be killed with probability $\mathfrak{p}_\Delta$. 

We refer to a discrete-time Markov chain whose transition matrix is of the form \eqref{eq:22} as a $\emph{boundary random walk}$ (BRW for short). Observe that the transition matrix of a BRW is completely determined by the vector $(\mathfrak{p}_\Delta, \mathfrak{p}_0,\cdots)$, which may equivalently be viewed as a probability measure on $\bN_\Delta$. For convenience, this measure is called the $\emph{jumping measure}$ (at $0$) of $X$.

% --- TikZ 图形代码开始 ---
\begin{figure}
    \centering
    \begin{center}
\begin{tikzpicture}[scale=0.7]

% 球 Δ 在最左边
\coordinate (ballDelta) at (0, 0);
\fill[blue!70] (ballDelta) circle (2pt);
\draw[black, thick] (ballDelta) circle (2pt);
\node[below, yshift=-4mm] at (ballDelta) {$\Delta$};

% 7 个小球序列右移
\begin{scope}[xshift=2.25cm]
    % 定义 7 个小球的位置
    \foreach \i in {0,1,2,3,4,5,6} {
        \coordinate (ball\i) at (\i*1.5, 0);
    }
    % 连接直线
    \draw[thick, gray!60] (ball0) -- (ball1) -- (ball2) -- (ball3) -- (ball4) -- (ball5) -- (ball6);
    % 绘制小球
    \foreach \i in {0,1,2,3,4,5,6} {
        \fill[red!70] (ball\i) circle (2pt);
        \draw[black, thick] (ball\i) circle (2pt);
        \node[below, yshift=-4mm] at (ball\i) {$\i$};
    }
    % 右侧总省略号
    \node[scale=1] at (9.9, 0) {$\cdots$};
\end{scope}

% 0→Δ 抛物线
\draw[thick, ->, >=stealth, red!80]
    (2.25,0.4) .. controls (1.6,2.0) and (0.65,2.0) .. (0,0.4);
\node[above, font=\small, fill=white, rounded corners=2pt, inner sep=2pt]
    at (1.125, 2.0) {$\fp_\Delta$};

% 0→1 抬高至顶点 2.0
\draw[thick, ->, >=stealth, blue!80]
    (2.25,0.4) .. controls (2.7,2.0) and (3.0,2.0) .. (3.75,0.4);
\node[above, font=\small, fill=white, rounded corners=2pt, inner sep=2pt]
    at (3.0, 2.0) {$\fp_1$};

% 0→2 抬高至顶点 1.7
\draw[thick, ->, >=stealth, green!80]
    (2.25,0.4) .. controls (3.0,1.7) and (4.0,1.7) .. (5.25,0.4);
\node[above, font=\small, fill=white, rounded corners=2pt, inner sep=2pt]
    at (3.75, 1.7) {$\fp_2$};

% 0→3 抬高至顶点 1.4
\draw[thick, ->, >=stealth, orange!80]
    (2.25,0.4) .. controls (3.5,1.4) and (5.0,1.4) .. (6.75,0.4);
\node[above, font=\small, fill=white, rounded corners=2pt, inner sep=2pt]
    at (4.5, 1.4) {$\fp_3$};

% 第一个省略号：夹在 p^N_3 与 p^N_5 之间
%\node[scale=1] at (7.6, 0.4) {$\cdots$};

% 0→5 再抬高一点（顶点 0.9）
\draw[thick, ->, >=stealth, purple!80]
    (2.25,0.4) .. controls (3.8,0.9) and (6.0,0.9) .. (8.25,0.4);
\node[above, font=\small, fill=white, rounded corners=2pt, inner sep=2pt]
    at (6.0, 0.9) {$\fp_4$};

% 第二个省略号：p^N_5 右侧
\node[scale=1] at (9, 0.8) {$\cdots$};

% 新增：从标号5到标号6的抛物线箭头（开口向上）
\draw[thick, ->, >=stealth, brown!80]
    (9.75,-0.4) .. controls (10.25,-1.2) and (10.75,-1.2) .. (11.25,-0.4);
\node[below, font=\small, fill=white, rounded corners=2pt, inner sep=2pt]
    at (10.5, -1.2) {$\dfrac{1}{2}$};

% 新增：从标号5到标号4的抛物线箭头（开口向上）
\draw[thick, ->, >=stealth, teal!80]
    (9.75,-0.4) .. controls (9.25,-1.2) and (8.75,-1.2) .. (8.25,-0.4);
\node[below, font=\small, fill=white, rounded corners=2pt, inner sep=2pt]
    at (9.0, -1.2) {$\dfrac{1}{2}$};

% 新增：从0下方出发的3/4圆环箭头
\draw[thick, ->, >=stealth, magenta!80] 
    (2.25,0.4) arc [start angle=45, end angle=315, radius=0.6];
\node[below right, font=\small, fill=white, rounded corners=2pt, inner sep=2pt]
    at (0.55, 0.25) {$\fp_0$};
    
\end{tikzpicture}
\end{center}
    \caption{Jumping behavior of BRW.}
    \label{fig:my_label}
\end{figure}
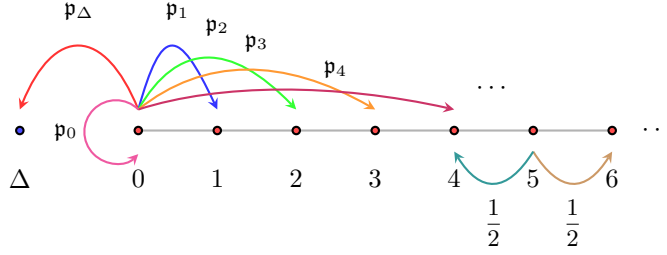
% --- TikZ 图形代码结束 ---

Let $\mathscr G^X_m:=\sigma(\{X_k: 0\leq k \leq m\})$ denote the sub-$\sigma$-algebra of $\mathscr G$ generated by $\{X_k: 0\leq k\leq m\}$, and let $(\mathscr G^X_m)_{m\geq 0}$ be the natural filtration of $X$. Note that any function $f$ defined on $\mathbb{N}$ is implicitly extended to $\mathbb{N}_\Delta$ by setting $f(\Delta):=0$. We now state a simple lemma for later use.

\begin{lemma}\label{LM21}
Assume that $\bP(X_0=i_0)=1$ for some $i_0\in \bN$, and that there exists an integer $j_0$ such that $\mathfrak{p}_j=0$ for all $j>j_0$.  Then, for any real-valued function $f$ on $\bN$, the process defined by
\[
M^f_0:=0,\quad M^f_m:=f(X_m)-f(X_0)-\sum_{k=0}^{m-1} (P-I)f(X_k),\quad m\geq 1,
\]
is a square-integrable $(\mathscr G^X_m)$-martingale, where $(P-I)f(i):=\sum_{j\in \bN}p_{ij}f(j)-f(i)$ for all $i\in \bN$. 
\end{lemma}
\begin{proof}
Note that for any $\omega\in \Omega$, we have $|X_m(\omega)|\leq i_0+mj_0$. Hence, $X_m$ is a bounded random variable for each $m$. It follows that $M^f_m$ is also bounded and therefore square-integrable. The conclusion then follows directly from the Markov property of $X$.
\end{proof}

Another auxiliary fact that will be needed later is that the series
\begin{equation}\label{eq:23-2}
F(x):=\sum_{k=0}^\infty \mathbb{P}(X_k=0)\, x^k,\quad 0\leq x<1,
\end{equation}
admits an explicit representation in terms of the jumping measure $(\mathfrak{p}_j)_{j\in\mathbb{N}_\Delta}$.

\begin{lemma}\label{LM22}
Assume that $\bP(X_0=0)=1$. Then the series $F(x)$ defined as \eqref{eq:23-2} admits the following explicit expression
\begin{equation}\label{eq:24-2}
F(x)=\frac{1}{1-x\cdot \sum_{k=0}^\infty \fp_k \left(\frac{1-\sqrt{1-x^2}}{x} \right)^k},\quad 0\leq x<1. 
\end{equation}
In particular, if $\fp_0=0$, then for any positive integer $m$, 
\begin{equation}\label{eq:25-2}
\sum_{k=0}^{m-1}\bP(X_k=0)\leq \frac{e}{\sqrt{1-e^{-\frac{2}{m}}}}. 
\end{equation}
\end{lemma}
\begin{proof}
Although the proof of \eqref{eq:24-2} relies only on elementary techniques, its detailed verification is nevertheless rather cumbersome. We therefore defer it to Appendix~\ref{APPA}. Let us prove the estimate \eqref{eq:25-2}. Taking $x:=e^{-1/m}\in (0,1)$, we have
 \[
     \sum_{k=0}^{m-1}\bP(X_k=0)\leq e\sum_{k=0}^{m-1}\bP(X_k=0)\, x^k\leq eF(x). 
 \]
Since $\fp_0=0$ and $\frac{1-\sqrt{1-x^2}}{x}\in (0,1)$, it follows that 
 \[
 \begin{aligned}
    F(x)&=\frac{1}{\sqrt{1-x^2}+(1-\sqrt{1-x^2})-(1-\sqrt{1-x^2})\sum_{k=1}^\infty\fp_k\left(\frac{1-\sqrt{1-x^2}}{x} \right)^{k-1}} \\
    &\leq \frac{1}{\sqrt{1-x^2}+(1-\sqrt{1-x^2})-(1-\sqrt{1-x^2})\sum_{k=1}^\infty\fp_k}\leq \frac{1}{\sqrt{1-x^2}}. 
 \end{aligned}\]
 Substituting $x = e^{-1/m}$ then yields the estimate \eqref{eq:25-2}.
\end{proof}
\begin{remark}
If $\fp_\Delta>0$, then \eqref{eq:24-2} remains valid at $x=1$. In particular, $\sum_{k=0}^\infty \bP(X_k=0)=1/\fp_{\Delta}$ when $X$ starts from the origin $0$. 
\end{remark}

\subsection{Feller's Brownian motion}\label{SEC22}

Let $E:=[0,\infty)$, and let $E_\Delta:=E\cup \{\Delta\}$ denote its one-point compactification, where $\Delta$ is identified with $\infty$ in the extended real number system endowed with the usual Euclidean topology. Moreover, let $C_0(E)$, equipped with the uniform norm $\|\cdot\|_u$, be the Banach space of continuous functions $f$ on $E$ satisfying $\lim_{x\rightarrow \infty}f(x)=0$. A Markov process on $E$ with $\Delta$ as the cemetery is called a $\emph{Feller process}$ if its transition semigroup $(P_t)_{t\geq 0}$ satisfies the following conditions: $P_0$ is the identity map on $C_0(E)$, $P_tC_0(E)\subset C_0(E)$ for all $t\geq 0$, and $P_tf\rightarrow f$ in $C_0(E)$ as $t\rightarrow 0$ for every $f\in C_0(E)$.

The definition of Feller's Brownian motion is provided below. (Although the sample space in this definition is denoted by the same symbol $\Omega$ as for the BRW, this mild abuse of notation will not cause any confusion in the subsequent context.)

\begin{definition}\label{DEF22}
Let $B=\left\{\Omega, \FF, \FF_t, B_t, \theta_t, (\bP^x)_{x\in E}\right\}$ be a Markov process on $E$ with lifetime $\zeta$. It is called a $\emph{Feller's Brownian motion}$ (FBM for short) if $B$ is a Feller process on $E$, and its killed process upon hitting $0$ has the same law as the killed Brownian motion on $(0,\infty)$.
\end{definition}

Let $B=(B_t)_{t\geq 0}$ be an FBM on $E$. Its transition semigroup $(P_t)_{t\geq 0}$ is fully determined by its $\emph{infinitesimal generator}$. Recall that a function $f\in C_0(E)$ belongs to the domain $\cD(\sL)$ of the infinitesimal generator if the limit
\[
\sL f:=\lim_{t\downarrow 0}\frac{P_t f-f}{t}
\]
exists in $C_0(E)$. The operator $\sL$ with domain $\cD(\sL)$ is called the infinitesimal generator of $(P_t)_{t\geq 0}$. The following characterization of $\sL$ for FBMs is due to Feller \cite{F52}; a detailed proof can also be found in \cite[Appendix~A]{Li25}.

\begin{theorem}\label{THM23}
A Markov process $B=(B_t)_{t\geq 0}$ on $E$ is an FBM if and only if it is a Feller process on $E$ whose infinitesimal generator on $C_0(E)$ is $\sL f=\frac{1}{2}f''$ with domain
\begin{equation}\label{eq:37}
\begin{aligned}
\cD(\sL)=&\bigg\{f\in C_0(E)\cap C^2({\color{blue}(}0,\infty)): f''\in C_0(E), \\
	&\quad p_1f(0)-p_2f'(0)+\frac{p_3}{2}f''(0)+\int_{(0,\infty)}\left(f(0)-f(x)\right)\,p_4(\dd x)=0\bigg\},
\end{aligned}\end{equation}
where $p_1,p_2,p_3$ are non-negative numbers, and $p_4$ is a non-negative measure on $(0,\infty)$ such that
\begin{equation}\label{eq:330}
	\int_{(0,\infty)}(x\wedge 1)\,p_4(\dd x)<\infty
\end{equation}
and
\begin{equation}\label{eq:38}
	|p_4|:=p_4\left((0,\infty) \right)=\infty \quad \text{if}\quad p_2=p_3=0.
\end{equation}
The parameters $p_1,p_2,p_3,p_4$ are uniquely determined up to a multiplicative constant for each FBM.
\end{theorem}
\begin{remark}\label{RM24}
It is worth noting that $p_4([\varepsilon,\infty))<\infty$ for any $\varepsilon>0$ in view of \eqref{eq:330}. In addition,  for any $f\in \cD(\sL)$, the derivative at the boundary $f'(0):=\lim_{x\downarrow 0}f'(x)$ exists, and $\lim_{x\uparrow \infty}f'(x)=0$. (The existence of the first limit can be found in, e.g., \cite[(3.12)]{Li25}, while the second limit follows from \cite[(3.6)]{Li25} together with the equality above \cite[(3.12)]{Li25}.) In particular, $\|f'\|_u=\sup_{x\in [0,\infty)}|f'(x)|<\infty$.  
\end{remark}

The four parameters $p_1, p_2, p_3, p_4$ introduced in Theorem~\ref{THM23} correspond, respectively, the four possible boundary behaviors of an FBM at $0$: \emph{killing}, \emph{reflecting}, \emph{sojourn}, and \emph{jumping}. Using these parameters, we provide an intuitive—though not fully rigorous—description of the boundary behavior of an FBM upon hitting 0, which serves to motivate the scaling strategy employed in  Theorems~\ref{THM31} and \ref{THM35}.
For a precise and comprehensive account, the reader is referred to \cite{IM63} or \cite[\S3]{Li25}.

Broadly speaking, the boundary behavior of an FBM at $0$ can be decomposed into two stages. The first stage consists of two types of continuous behavior: sojourn, in which the process remains at $0$ for a positive duration of time (if $p_3>0$), and reflection, in which the process repeatedly moves slightly to the right of $0$, returns to $0$, and then bounces out again within an extremely short time interval (if $p_2>0$). Each departure from $0$ to the right, followed by a return to $0$, is called an \emph{excursion}. In this stage, the boundary behavior is governed primarily by a sequence of such excursions, while the remaining sojourn times at $0$ are broken into numerous tiny disjoint pieces interspersed among the excursions, forming a structure reminiscent of a Cantor-type set. The first stage ends abruptly either when the process is about to return to $0$ at the completion of an excursion (when $p_2>0$) or after it has spent a positive amount of time at $0$ (when $p_2=0$ and $p_3>0$). At that point, the boundary behavior transitions into the second stage, namely the jump stage: the FBM jumps to a neighborhood of $x>0$ with `probability'
$p_4(\mathrm{d}x)/(p_1 + |p_4|)$,
or is killed (that is, jumps to the cemetery $\Delta$) with `probability' $p_1/(p_1 + |p_4|)$. For convenience, we call
\begin{equation}\label{eq:26}
\mu:=\frac{p_1\cdot \delta_\Delta+p_4}{p_1+|p_4|}
\end{equation}
the \emph{jumping measure} (on $E_\Delta$) of the FBM at $0$. 
It should be noted that when $|p_4| = \infty$, the jumping measure defined above ceases to be meaningful. In this case, the jump behavior in the second stage becomes highly intricate; a rigorous characterization is provided in \cite[\S12–16]{IM63}. 

The key point is to examine the termination time of the first stage, that is, the time at which the jump occurs. This time consists of two components: the total sojourn time and the total reflection time. The sojourn time is relatively straightforward to describe; it follows an exponential distribution with parameter $(p_1 + |p_4|)/p_3$ and can be expressed as
\[
\xi_1 := \frac{p_3}{p_1 + |p_4|} \, \mathfrak{e},
\]
where $\mathfrak{e}$ is an exponential random variable with parameter $1$. The reflection time $\xi_2$, in contrast, involves the local time $(L_t)_{t\ge 0}$ of the reflected Brownian motion on $E$ at $0$ (see \cite[Chapter V, \S2]{RY99} for the definition), and can be written as
\[
\xi_2 := \inf \left\{ t>0 : L_t > \frac{p_2}{p_1 + |p_4|} \, \mathfrak{e} \right\}.
\]
A more rigorous formulation for these two times is given in \cite[17b]{IM63}.
We now provide a rough, average estimate for them. Clearly, the expected sojourn time is simply the expectation of $\xi_1$, which is equal to $p_3 / (p_1 + |p_4|)$. According to the L\'evy identity (see \cite[Chapter V, Theorem 2.3]{RY99}) and the reflection principle (see \cite[Chapter III, Proposition 3.7]{RY99}), $L_t$ grows roughly like $t^{1/2}$. Therefore, we may approximate
\[
    \xi_2 \approx \left( \frac{p_2}{p_1 + |p_4|} \right)^2.
\]

In summary, each time the FBM reaches $0$, it remains in a sojourn at $0$, undergoes reflection near $0$, or exhibits both behaviors simultaneously for approximately
 \begin{equation}\label{eq:27}
   \xi= \xi_1 + \xi_2 \approx \frac{p_3}{p_1 + |p_4|} + \left( \frac{p_2}{p_1 + |p_4|} \right)^2
 \end{equation}
 units of time, after which it jumps away according to the `probability' given by \eqref{eq:26}. If the process is not killed, it then evolves as a standard Brownian motion until it returns to $0$, at which point the same boundary behavior is enacted, and this cycle repeats indefinitely.

% For a more detailed discussion of these boundary behaviors, see \cite{IM63} or \cite[\S3]{Li25}.

\section{From BRWs to FBM: Heuristic investigations}\label{SEC3}

Let $n$ be a positive integer, and let $X=(X_k)_{k\geq 0}$ be a BRM with transition matrix given by \eqref{eq:22}. Motivated by the classical Donsker's invariance principle, we consider the rescaled process
\[
    B^{(n)}_t:=\frac{1}{n}X_{\lfloor n^2t\rfloor},\quad t\geq 0, 
\]
where $\lfloor \cdot \rfloor$ denotes the floor function. 
The primary goal of this paper is to construct a suitable BRW such that, for sufficiently large $n$, the rescaled process $B^{(n)}$ approximates a given FBM.

\subsection{Power-law scaling schemes for approximating FBM by BRWs}

Although the waiting time between successive jumps of $B^{(n)}$ is $1/n^{2}$, its jump mechanism closely mirrors that of $X$, with the distinction that the coefficient $\fp_j$ in \eqref{eq:22} now determines the probability of jumping from $0$ to the point $j/n$. This leads to the following heuristic observation. 

Fix $\varepsilon>0$. For sufficiently large $n$, any jump of $B^{(n)}$ away from $0$ will land in $(0,\varepsilon)$ with probability at least
\[
\sum_{j=1}^{\lfloor n\varepsilon\rfloor} \fp_j.
\]
If the coefficients $\fp_j$ were independent of $n$, then as $n\rightarrow\infty$ this sum would converge to the total probability of making a genuine jump away from $0$ (that is, excluding staying at $0$ or being killed). Consequently, almost all actual jumps of $B^{(n)}$ would occur arbitrarily close to $0$. This stands in sharp contrast with the FBM, whose boundary jumps are governed by the measure $p_4(\mathrm{d}x)$ and therefore retain a nontrivial probability of landing far from $0$. Thus, ensuring compatibility between the boundary jumping behavior of $B^{(n)}$ and that of the limiting FBM necessitates that the coefficients $\fp_j$ depend on $n$ and, for most indices $j$, vanish as $n \rightarrow \infty$.

%Thus, in order for the boundary jumping behavior of $B^{(n)}$ to be compatible with that of the FBM it is intended to approximate, the coefficients $\fp_j$ must necessarily depend on $n$ and become small as $n$ grows large for most $j$.

This observation indicates that, slightly different from the classical Donsker invariance principle, the BRW itself should depend on $n$ when performing the rescaling. We denote it by $X^{(n)}$, with transition matrix $P^{(n)}$ and jumping measure $(\mathfrak{p}_j^{(n)})_{j\in \mathbb{N}_\Delta}$, respectively. 
For the sake of rigor, we rewrite $B^{(n)}$ as
\begin{equation}\label{eq:33}
    B^{(n)}_t:=\frac{1}{n}X^{(n)}_{\lfloor n^2t\rfloor},\quad t\geq 0. 
\end{equation}
%It should be noted that this minor modification does not affect the discussion in \S\ref{SEC31}, so there is no need to change the notations used there.

%\begin{remark}\label{RM32}
Based on the intuitive description of the FBM's boundary behavior given after Theorem~\ref{THM23}, we now examine how the jumping measure $(\fp^{(n)}_j)_{j\in \mathbb{N}_\Delta}$ of $B^{(n)}$ should be specified so that its behavior near the boundary closely parallels that of an FBM with parameters  $(p_1, p_2, p_3, p_4)$.
In the first stage, the boundary behavior of $B^{(n)}$ similarly consists of sojourn and reflection components. Here, the probability $\fp^{(n)}_0$ of remaining at $0$ plays the role of the sojourn parameter, while the probability $\fp^{(n)}_1$ of jumping from $0$ to the adjacent point $1/n$ corresponds to $p_2$ and serves as the reflection parameter for $B^{(n)}$: the segment of the trajectory that starts at $0$, jumps to $1/n$, then evolves as a symmetric random walk, and finally returns to $0$, constitutes a single `excursion' of $B^{(n)}$ at the boundary.
The remaining probabilities $(\fp^{(n)}_j)_{j\ge 2}$ collectively represent jumps from $0$ into $(0,\infty)$, while $\fp^{(n)}_\Delta$ accounts for killing. Together, they capture the second stage of the boundary behavior.
If the behavior of $B^{(n)}$ closely approximates that of the FBM, we can estimate the expected duration of its first-stage boundary behavior:
\begin{equation}\label{eq:37-2}
    \xi^{(n)}\approx \frac{1}{n^2}\left[\frac{\fp^{(n)}_0}{\fp^{(n)}_\Delta+\sum_{j\geq 2}\fp^{(n)}_j}+\left( \frac{\fp^{(n)}_1}{\fp^{(n)}_\Delta+\sum_{j\geq 2}\fp^{(n)}_j}\right)^2 \right],
\end{equation}
analogously to the calculation in \eqref{eq:27}, taking into account that the unit time for $B^{(n)}$ (that is, the waiting time between successive jumps) is $1/n^2$. 

We further consider that the family $(\fp^{(n)}_j)_{j\in \bN_\Delta}$ admits a power-law asymptotic representation as $n\rightarrow \infty$, motivated by the observation made at the beginning of this subsection that the coefficients $\fp^{(n)}_j$ (for most $j$) must vanish as $n\to\infty$. In particular, the coefficients $\fp^{(n)}_j$ for $j\ge 2$ and the killing probability $\fp^{(n)}_\Delta$, which describe the second-stage jump behavior at the boundary, should decay at a comparable rate; otherwise, terms with faster decay would disappear in the limit and could be set equal to zero from the outset without affecting the limiting dynamics. Accordingly, we postulate that
\begin{equation}\label{eq:38-2}
\fp^{(n)}_0 \approx \frac{\mathfrak{c}_0}{n^\alpha}, \quad \fp^{(n)}_1 \approx \frac{\fc_1}{n^\beta}, \quad \fp^{(n)}_j \approx \frac{\fc_j}{n^\gamma},\; j\geq 2, \quad \fp^{(n)}_\Delta \approx \frac{\fc_\Delta}{n^\gamma},
\end{equation}
for some constants $\alpha, \beta\geq 0$ and $\gamma>0$ depending not on $n$. Here, $(\fc_j)_{j \ge 2}$ form a jumping measure close to \eqref{eq:26}, essentially obtained by discretizing $p_4$ (see \eqref{eq:311}); but each $\fc_j$ inevitably depends on $n$, since it represents the jump to the point $j/n$ that varies with $n$. The positivity of exponents $\alpha$ and $\beta$ is required to ensure that 
\begin{equation}\label{eq:39}
    \sum_{j\in \bN_\Delta} \fp^{(n)}_j = 1
\end{equation}
 as $n$ grows large. Substituting \eqref{eq:38-2} into \eqref{eq:37-2}, we have
 \begin{equation}\label{eq:312}
    \xi^{(n)}\approx \frac{\fc_0 \cdot n^{\gamma-\alpha-2}}{\fc_\Delta+\sum_{j\geq 2}\fc_j}+\left( \frac{\fc_1\cdot  n^{\gamma-\beta-1}}{\fc_\Delta+\sum_{j\geq 2}\fc_j}\right)^2.
\end{equation}
The expression for $\xi^{(n)}$ should be consistent with \eqref{eq:27}, and combined with the constraint \eqref{eq:39}, this places strong restrictions on the admissible values of the exponents $\alpha,\beta,\gamma$. Only two regimes are possible:
\begin{itemize}
\item[(1)] When $p_3\neq 0$ (hence $\fc_0\neq 0$), one must have $\alpha=0, \beta=1$, and $\gamma=2$. In this case, if $p_2=0$ then we may take $\fc_1=0$, so the value of $\beta$ is irrelevant.
\item[(2)] When $p_3=0$ and $p_2\neq 0$, we may take $\fc_0=0$, making $\alpha$ immaterial, while the other exponents must be $\beta=0$ and $\gamma=1$.
\end{itemize}
A careful comparison between \eqref{eq:312} and \eqref{eq:27} offers a clear heuristic guide for formulating a rigorous version of the asymptotic relation \eqref{eq:38-2} in terms of the parameters $p_1,p_2,p_3,p_4$. When $p_3\neq 0$, the precise asymptotic representation is given by \eqref{eq:311} and \eqref{eq:313}, while the second case is presented in \eqref{eq:316} and \eqref{eq:317}. For both cases, the corresponding convergence results will be formulated rigorously in the subsequent sections. %, where complete proofs are also provided.

When proving the convergence of $B^{(n)}$ to $B$ for these two cases, several difficulties arise. In the first case $p_3\neq 0$, the main difficulty lies in establishing the tightness of $B^{(n)}$. This issue is resolved by an application of Aldous’ tightness criterion, combined with suitable martingale estimates based on Lemma~\ref{LM21}. We emphasize that the key estimates and computations in this argument are carried out at the boundary point $0$; see in particular \eqref{eq:410}, \eqref{eq:419}, and \eqref{eq:515}. 
In the second case, where $p_2 \neq 0$ and $p_3 = 0$, even greater obstacles arise. In this case, the normalization in the definition of the jumping measure in \eqref{eq:316} involves a denominator of order $n$, rather than $n^2$ as in the previous case. As a consequence, the boundary estimates \eqref{eq:410} and \eqref{eq:419}, which were crucial in the first case, are no longer sufficient to close the argument. To overcome this difficulty, we make essential use of Lemma~\ref{LM22}. Since the process $X^{(n)}$ does not spend more than one unit of time at the boundary once it hits $0$ (that is, $\fp^{(n)}_0 = 0$), the total amount of time that $X^{(n)}$ spends at $0$ up to time $n^2 t$ is of order $\sqrt{n^2 t}$; see Lemma~\ref{LM81}. This yields an effective gain of one factor $n$ compared with the first case, precisely compensating for the missing factor $n$ in the denominator of \eqref{eq:316}.

\subsection{The case $p_2=p_3=0$}\label{SEC33}

It is also worth noting that, for the remaining case $p_2 = p_3 = 0$, one necessarily has $|p_4| = \infty$ by \eqref{eq:38}. The corresponding Feller’s Brownian motion does not appear to be amenable to approximation by a sequence $B^{(n)}$ whose jumping measures follow the power-law form \eqref{eq:38-2}, since imposing \eqref{eq:38-2} with \(\fc_0 = \fc_1 = 0\) seems incompatible with \eqref{eq:39}. In this setting, a more natural scaling scheme would be
\[
    \fp^{(n)}_\Delta=\frac{p_1}{p_1+p_4((1/n,n])},\quad \fp^{(n)}_j=\frac{p_4((j/n,(j+1)/n])}{p_1+p_4((1/n,n])},\;1\leq j\leq n^2-1,
\]
with $\fp^{(n)}_j:=0$ for $j=0$ or $j\geq n^2$. 

However, the fundamental difficulty of this approach lies in the fact that, as observed in Remark~\ref{RM45}, the growth rate of $p_4((1/n,n])$ as $n \rightarrow \infty$ is strictly slower than linear in $n$. From a technical standpoint, this precludes the direct adaptation of the argument used for the second case $p_2 \neq 0,\, p_3 = 0$. In particular, we cannot establish the tightness of $B^{(n)}$ under the aforementioned scaling. We therefore leave this case for future investigation.

%It is also worth noting that, for the remaining case $p_2=p_3=0$, we must have $|p_4|=\infty$ by \eqref{eq:38}. The corresponding FBM does not seem amenable to approximation by a sequence $B^{(n)}$ with jumping measures of the power-law form \eqref{eq:38-2}, since \eqref{eq:38-2} with $\fc_0 = \fc_1 = 0$ appears to be in conflict with \eqref{eq:39}. A more natural scaling scheme should beand 

%, and it is unclear whether these are merely technical in nature or reflect a deeper obstruction. (For instance, because each boundary interaction of $B^{(n)}$ occurs only after a sojourn of length $1/n^2$ at $0$, might the limiting process necessarily exhibit a nontrivial sojourn at the boundary?) We hope to investigate this issue further in future work.

\section{From BRWs to FBM: Rigorous results}\label{SEC4}

Based on the analysis conducted in the preceding section, we present the main result of this paper, with detailed proofs deferred to the subsequent sections.

\subsection{Realization on the Skorohod space}\label{SEC31}

The rescaled process defined in \eqref{eq:33} is a continuous-time stochastic process (though not a Markov chain) with state space
\[
E^{(n)}:=\frac{1}{n}\bN=\left\{0,\frac{1}{n}, \frac{2}{n},\cdots\right\}. 
\]
The cemetery for $B^{(n)}$ is still denoted by $\Delta$, and $E^{(n)}_\Delta:=E^{(n)}\cup \{\Delta\}$ stands for the one-point compactification of $E^{(n)}$.
Clearly, each sample path $[0,\infty)\ni t\mapsto B^{(n)}_t(\omega)$ of $B^{(n)}$ is a c\`adl\`ag function on $E^{(n)}_\Delta(\subset E_\Delta)$, and hence also a c\`adl\`ag function on $E_\Delta$. This shows that every $B^{(n)}$ can be realized as a process on the Skorohod space. 

More precisely, we denote the Skorohod space by
\[
D_{E_\Delta}:=D_{E_\Delta}([0,\infty))
\]
equipped with the Skorohod topology (see, e.g., \cite[Chapter 3]{B99}), which consists of all c\`adl\`ag functions from $[0,\infty)$ to $E_\Delta$. Let $(\pi_t)_{t\geq 0}$ be the canonical projection family on $D_{E_\Delta}$; that is, for $t\geq 0$ and $w\in D_{E_\Delta}$, $\pi_t (w):=w(t)$. Then $\mathscr D_t:=\sigma(\{\pi_s: 0\leq s\leq t\})$, $t\geq 0$, forms the natural filtration on $D_{E_\Delta}$, and $\mathscr D:=\sigma(\{\pi_t:t\geq 0\})$ coincides with the Borel $\sigma$-algebra on $D_{E_\Delta}$ generated by the Skorohod topolopgy (see, e.g., \cite[Theorem~16.6]{B99}). 

Recall that $B^{(n)}$ is defined on the probability space $(\Omega, \mathscr G,\bP)$ associated with $X^{(n)}$. (Strictly speaking, this probability space may depend on $n$; however, since this dependence plays no essential role in the subsequent discussion, we suppress the index $n$ for simplicity).  Let $\sG^{(n)}_t:=\sigma(\{B^{(n)}_s:0\leq s\leq t\})$, $t\geq 0$, denote the natural filtration of $B^{(n)}$. It is straightforward to verify that
\begin{equation}\label{eq:35}
\sG^{(n)}_t=\sG^{X^{(n)}}_m,\quad \forall  t\in [m/n^2,(m+1)/n^2), m\geq 0,
\end{equation}
where $(\sG^{X^{(n)}}_m)_{m\geq 0}$ is the natural filtration of $X^{(n)}$. Consider the map
\[
\mathcal B^{(n)}: \Omega\rightarrow D_{E_\Delta}, \quad \omega\mapsto B^{(n)}_\cdot (\omega),
\]
where $B^{(n)}_\cdot(\omega)$ denotes the c\`adl\`ag function $w^{(n)}(t):=B^{(n)}_t(\omega)$ for $t\geq 0$. Note that $\mathcal B^{(n)}$ is both $\sG/\sD$- and $\sG^{(n)}_t/\sD_t$-measurable. In particular, the image measure of $\bP$
\begin{equation}\label{eq:34}
\bQ^{(n)}(\Lambda):=\bP\left(\left\{\omega\in \Omega: \cB^{(n)}(\omega)\in \Lambda\right\}\right),\quad \forall \Lambda\in \sD
\end{equation}
gives a probability measure on $(D_{E_\Delta}, \sD)$, and the corresponding canonical process on the Skorohod space $D_{E_\Delta}$, i.e., the following collection
\[
\left\{D_{E_\Delta},\sD,\sD_t, \pi_t,\bQ^{(n)}\right\},
\]
has the same finite-dimensional distributions as $B^{(n)}$. We refer to $\bQ^{(n)}$ as the \emph{law} of $B^{(n)}$ on the Skorohod space. 

Since every Feller process has a.s. c\`adl\`ag sample paths, the FBM described in Definition~\ref{DEF22} can likewise be realized on the Skorohod space $D_{E_\Delta}$. More precisely, for each $x\in E$, the FBM starting from $x$ can be represented as a law $\bQ^x:=\bP^x\circ \mathcal{B}^{-1}$ on $D(E_\Delta)$ via the map
\[
\mathcal B: \Omega\rightarrow D_{E_\Delta}, \quad \omega\mapsto B_\cdot (\omega). 
\]

In what follows, when the starting point $x \in E$ of the FBM is fixed, we write $\bQ^{(n)} \Rightarrow \bQ^{x}$  to indicate that $\bQ^{(n)}$ converges weakly to $\bQ^{x}$ on the Skorohod space $D_{E_\Delta}$. This means that for any bounded continuous function $F$ on $D_{E_\Delta}$, 
\[
\lim_{n\rightarrow\infty}\int_{D_{E_\Delta}}F(w)\bQ^{(n)}(\dd w)=\int_{D_{E_\Delta}}F(w)\bQ^x(\dd w).
\]
By the definition of the image measures $\bQ^{(n)}$ and $\bQ^{x}$, this is equivalent to
\[
    \lim_{n\rightarrow \infty} \bE\left[F(\mathcal{B}^{(n)})\right]=\bE^x \left[F(\mathcal B)\right]. 
\]
Thus, even though $B^{(n)}$ and $B$ are not defined on the same probability space, we say that $B^{(n)}$ converges weakly to $B$, denoted by $B^{(n)} \Rightarrow B$, if the law of $B^{(n)}$ on $D_{E_\Delta}$ converges weakly to that of $B$. This terminology and notation will also be used for other sequences of c\`adl\`ag processes.

\begin{remark}\label{RM31}
Two c\`adl\`ag processes are said to be \emph{equivlant} if their laws on the Skorohod space coincide. Equivalently, this means that they have the same family of finite-dimensional distributions.
It is important to emphasize that weak convergence depends solely on the laws of the processes and is independent of the underlying probability spaces on which they are realized. In particular, if $B^{(n)}$ and $\tilde{B}^{(n)}$ share the same law, then $B^{(n)} \Rightarrow B$ automatically implies $\tilde{B}^{(n)} \Rightarrow B$.
\end{remark}

\subsection{Main result for $p_3\neq 0$}

The first case, $p_3\neq 0$, is handled as follows. The proof will be completed in Sections \ref{Se4} and \ref{SEC5}.

\begin{theorem}\label{THM31}
Let $B$ be an FBM with parameters $(p_1,p_2,p_3,p_4)$ as specified in Theorem~\ref{THM23}, with $p_3\neq 0$. 
For any $x_0\in E$ and $n\geq 2$, consider the BRW $X^{(n)}$  starting from $\lfloor nx_0\rfloor$, whose jumping measure is given by
\begin{equation}\label{eq:311}
\begin{aligned}
&\fp^{(n)}_\Delta:=\frac{p_1}{n^2p_3},\quad \fp^{(n)}_1:=\frac{p_2}{np_3}, \\
&\fp^{(n)}_j:=\frac{p_4\left((j/n,(j+1)/n]\right)}{n^2p_3},\; 2\leq j\leq n^2-1,\quad \fp^{(n)}_j:=0,\; j\ge n^2,
\end{aligned}\end{equation}
and 
\begin{equation}\label{eq:313}
\fp^{(n)}_0:=1-\fp^{(n)}_\Delta-\sum_{j=1}^\infty\fp^{(n)}_j. 
\end{equation}
Let $B^{(n)}$ denote the rescaled process of $X^{(n)}$ as in \eqref{eq:33}. Then, as $n\rightarrow \infty$, $B^{(n)}$ converges weakly to the FBM $B$ starting from $x_0$, i.e., $B^{(n)}\Rightarrow B$. 
\end{theorem}
\begin{remark}\label{RM33} 
\begin{itemize}
%\item[(1)] In light of \eqref{eq:330}, the technical condition \eqref{eq:36} can be weakened to the following: for some $K>0$,
%\[
%\int_{(K,\infty)} x^2 \, p_4(\mathrm{d}x) < \infty.
%\]
%This condition will be used only once, specifically in \eqref{eq:419}, when proving the Aldous' condition \eqref{eq:42}. 
\item[(1)] In \eqref{eq:311}, we set $\fp^{(n)}_j = 0$ for all $j \geq  n^2$ in order to facilitate the application of Lemma~\ref{LM21}. This truncation does not fundamentally affect the approximation scheme given by the power-law asymptotics in \eqref{eq:38-2}.
\item[(2)] Strictly speaking, $\fp^{(n)}_0$ defined as \eqref{eq:313} may initially be negative. However, we note that
    \[
        \sum_{j\geq 2}\fp^{(n)}_j=\frac{1}{n^2p_3}p_4((2/n,n])\leq \frac{\int_{(0,\infty)}(1\wedge x)\, p_4(\dd x)}{2np_3},
    \]
    which implies $\lim_{n\rightarrow \infty}\fp^{(n)}_0=1$. Consequently, for sufficiently large $n$, the possibility that $\fp^{(n)}_0 < 0$ no longer arises.
\end{itemize}
\end{remark}

\iffalse
 Now let us explain how, in a certain sense, the technical condition \eqref{eq:36} can be relaxed. To this end, we consider an increasing sequence of positive integers $\{k_n: n \ge 1\}$. For each $n\geq 1$, we define a jumping measure $(\widetilde \fp_j^{(n)})_{j\in \bN_\Delta}$ in terms of $k_n$ (replacing $n$) as follows:
\begin{equation}\label{eq:314}
\begin{aligned}
&\widetilde \fp^{(n)}_\Delta:=\frac{p_1}{k_n^2p_3},\quad \widetilde\fp^{(n)}_1:=\frac{p_2}{k_np_3}, \\
&\widetilde \fp^{(n)}_j:=\frac{p_4\left((j/k_n,(j+1)/k_n]\right)}{k_n^2p_3},\; 2\leq j\leq n k_n-1,\quad \widetilde\fp^{(n)}_j:=0,\; j\ge nk_n,
\end{aligned}\end{equation}
and 
\begin{equation}\label{eq:315}
\widetilde\fp^{(n)}_0:=1-\widetilde\fp^{(n)}_\Delta-\sum_{j=1}^\infty\widetilde\fp^{(n)}_j. 
\end{equation}
Then the following corollary holds without the condition \eqref{eq:36}. Its proof will be provided in Section~\ref{SEC6}. 

\begin{corollary}\label{COR35}
Let $B$ be an FBM starting from $x_0\in E$ with the parameters $(p_1,p_2,p_3,p_4)$ clarified in Theorem~\ref{THM23} such that $p_3>0$. Then there exists an increasing sequence of positive integers $\{k_n:n\geq 1\}$ such that the rescaled process 
\[
    \widetilde{B}^{(n)}_t:=\frac{1}{k_n}\widetilde{X}^{(n)}_{\lfloor k_n^2t\rfloor},\quad t\geq 0,
\]
in terms of $k_n$, converges weakly to $B$, i.e., $\widetilde{B}^{(n)}\Rightarrow B$, where the BRW $\widetilde X^{(n)}$  starts from $\lfloor k_nx_0\rfloor$ and has the jumping measure given by \eqref{eq:314} and \eqref{eq:315}, as $n\rightarrow \infty$. 
\end{corollary}
\fi

\subsection{Main result for $p_2\neq 0,p_3=0$}

The other case $p_2\neq 0,p_3=0$ can be addressed analogously as follows. The proof is presented in Section~\ref{SEC7}. 

\begin{theorem}\label{THM35}
    Let $B$ be an FBM with parameters $(p_1,p_2,p_3,p_4)$ as specified in Theorem~\ref{THM23}, with $p_2\neq 0$ and  $p_3=0$. 
%Assume that the following limit exists:
%\begin{equation}\label{eq:316-2}
%    \lim_{n\rightarrow \infty}\frac{p_4((2/n,1])}{n}. 
%\end{equation}
For any $x_0\in E$ and $n\geq 2$, consider the BRW $X^{(n)}$  starting from $\lfloor nx_0\rfloor$, whose jumping measure is given by
\begin{equation}\label{eq:316}
\begin{aligned}
&\fp^{(n)}_\Delta:=\frac{p_1}{np_2},\quad \fp^{(n)}_0:=0, \\
&\fp^{(n)}_j:=\frac{p_4\left((j/n,(j+1)/n]\right)}{np_2},\; 2\leq j\leq n^2-1,\quad \fp^{(n)}_j:=0,\; j\geq n^2,
\end{aligned}\end{equation}
and 
\begin{equation}\label{eq:317}
\fp^{(n)}_1:=1-\fp^{(n)}_\Delta-\sum_{j=2}^\infty\fp^{(n)}_j. 
\end{equation}
Let $B^{(n)}$ denote the rescaled process of $X^{(n)}$ as in \eqref{eq:33}. Then, as $n\rightarrow \infty$, $B^{(n)}$ converges weakly to the FBM $B$ starting from $x_0$, i.e., $B^{(n)}\Rightarrow B$. 
\end{theorem}
\begin{remark}\label{RM45}
%In a sense, the additional technical requirement \eqref{eq:316-2} arises from the specific choice of the spatial scaling $E_n=\tfrac{1}{n}E$. If instead one adopts a more rapidly refining spatial scaling scheme, for instance $E'_n := \frac{1}{2^n}E$, then the analogue of condition \eqref{eq:316-2} is automatically fulfilled and the corresponding limit is equal to $0$. %(This also explains why the constant $c$ does not enter the statement of Corollary~\ref{COR37}.) Indeed, by Stolz's theorem, we have:
%\begin{equation}\label{eq:410-2}
 %  \lim_{n\rightarrow \infty}\frac{p_4((2/2^n,1])}{2^{n}}=\lim_{n\rightarrow \infty}\frac{p_4((2/2^{n+1},2/2^n])}{2^{n}}\leq \lim_{n\rightarrow \infty}\int_{(2/2^{n+1},2/2^n]}x\,p_4(\dd x)=0. 
%\end{equation}
%In particular, under the present assumption that the limit \eqref{eq:316-2} exists, it must be equal to $0$ (since $\{1/2^n\}$ is a subsequence of $\{1/n\}$).
%The existence of \eqref{eq:316-2} is assumed to guarantee that $\lim_{n\rightarrow \infty}\fp^{(n)}_1$ exists. Indeed, a straightforward computation shows that 
%\[  \fp^{(n)}_\Delta+\sum_{j=2}^\infty\fp^{(n)}_j=\frac{p_1+p_4((2/n,n])}{np_2},
%\]
%which implies, by the existence of \eqref{eq:316-2}, that 
%\begin{equation}\label{eq:319-2}
 %   \lim_{n\rightarrow \infty}\fp^{(n)}_1=1-\lim_{n\rightarrow \infty}\frac{p_1+p_4((2/n,n])}{np_2}=1.
%\end{equation}
As in the situation described in Remark~\ref{RM33}~(2), we likewise have $\fp^{(n)}_1>0$ for all sufficiently large $n$. More precisely, the following limit holds:
\begin{equation}\label{eq:319-2}
     \lim_{n\rightarrow \infty}\fp^{(n)}_1=1.
\end{equation}
To verify this, we first note that by an integration-by-parts formula (see, e.g., \cite[Theorem~3.36]{F99}), for any $\varepsilon>0$,
\[
    \int_{(\varepsilon,1]}x\,p_4(\dd x)=\int_{(\varepsilon,1]}p_4((x,1])\dd x+\varepsilon \cdot p_4((\varepsilon,1]). 
\]
Since $\int_{(0,1]}x\,p_4(\dd x)=\lim_{\varepsilon\downarrow 0}\int_{(\varepsilon,1]}x\,p_4(\dd x)<\infty$, it follows that the limit $$\lim_{\varepsilon\downarrow 0}\varepsilon \cdot p_4((\varepsilon,1])$$ exists. We next observe that, by Stolz's theorem, 
\[
   \lim_{n\rightarrow \infty}\frac{p_4((1/2^n,1])}{2^{n}}=\lim_{n\rightarrow \infty}\frac{p_4((1/2^{n+1},1/2^n])}{2^{n}}\leq \lim_{n\rightarrow \infty}2\int_{(1/2^{n+1},1/2^n]}x\,p_4(\dd x)=0. 
\]
Consequently, $\lim_{\varepsilon\downarrow 0}\varepsilon \cdot p_4((\varepsilon,1])=0$. In particular, 
\[
    \lim_{n\rightarrow \infty}\fp^{(n)}_1=1-\lim_{n\rightarrow \infty}\frac{p_1+p_4((2/n,n])}{np_2}=1.
\]
This establishes \eqref{eq:319-2}. 
\end{remark}

\subsection{The special case $p_4=0$}

Let us examine the special case $p_4=0$.  %, where all the technical conditions in Theorems~\ref{THM31} and \ref{THM35} are not required. 
Our main results encompass the Donsker invariance principle for the classical models described below.
When $p_3\neq 0$, we can use $B^{(n)}$ defined as \eqref{eq:33} with 
    \[
    \fp^{(n)}_\Delta=\frac{p_1}{n^2p_3},\quad \fp_0^{(n)}=1-\frac{p_1}{n^2p_3}-\frac{p_2}{np_3},\quad \fp_1^{(n)}=\frac{p_2}{np_3},\quad \fp_j^{(n)}=0, \; j\geq 2,
    \]
    to approximate the following special FBMs:
    \begin{itemize}
        \item[(1a)] \emph{absorbed Brownian motion} for $p_1=p_2=0$, which, once hitting 0, remains there forever. 
        \item[(1b)] \emph{exponential holding Brownian motion} for $p_1\neq 0, p_2=0$. (The terminology is borrowed from \cite{EFJP24}.) This FBM, upon hitting $0$, remains there for an exponentially distributed sojourn time with mean $p_3/p_1$, after which it is killed. See \cite[\S9]{IM63} or \cite[\S3.2.1]{Li25} for more details about this process. 
        \item[(1c)] \emph{sticky Brownian motion} for $p_1=0, p_2\neq 0$, which is the time-changed process of reflected Brownian motion with speed measure $\frac{p_3}{2p_2}\delta_0+\lambda$, where $\lambda$ stands for the Lebesgue measure on $[0,\infty)$. See \cite[Corollary~3.5]{Li25} for more descriptions about this characterization.
        \item[(1d)] \emph{mixed Brownian motion} for $p_1,p_2\neq 0$. This terminology is also borrowed from \cite{EFJP24}. 
    \end{itemize}
    When $p_3=0$ and $p_2\neq 0$, the process $B^{(n)}$ with
 \[
    \fp^{(n)}_\Delta=\frac{p_1}{np_2},\quad \fp_0^{(n)}=0,\quad \fp_1^{(n)}=1-\frac{p_1}{np_2},\quad \quad \fp_j^{(n)}=0, \; j\geq 2,
    \]
    may be utilized to approximate the special FBMs as below:
    \begin{itemize}
        \item[(2a)] \emph{reflected Brownian motion} for $p_1=0$, whose infinitesimal generator is the \emph{Neumann Laplacian} (that is, the Lapalacian subject to the boundary condition $f'(0)=0$), multiplied by the factor $1/2$.
        \item[(2b)] \emph{elastic Brownian motion} for $p_1\neq 0$, whose infinitesimal generator is the \emph{Robin Laplacian} (that is, the Lapalacian subject to the boundary condition $f'(0)=\frac{p_1}{p_2}f(0)$), multiplied by the factor $1/2$.
    \end{itemize}   
    
Formally speaking, our result does not cover the invariance principle for killed Brownian motion (whose infinitesimal generator is the \emph{Dirichlet Lapalacian}, i.e., the Laplacian subject to the Dirichlet boundary condition $f(0)=0$, multiplied by $1/2$), since the latter is a Feller process on $(0,\infty)$ rather than on $[0,\infty)$. Nevertheless, if one interprets the absorbing boundary $0$ in case (1a) as the cemetery point $\Delta$, then the transformed `absorbed Brownian motion' has the same law as the killed Brownian motion. Concretely, let $X^{(n)}$ be a BRW on $\mathbb{N}\setminus\{0\}$ which, at the boundary site $1$, jumps to the site $2$ with probability $1/2$ and is killed with probability $1/2$ (this $X^{(n)}$ is in fact independent of $n$). The associated processes $B^{(n)}$ then converge weakly, as $n\to\infty$, to the killed Brownian motion on $(0,\infty)$.

%\subsection{Sketch of proofs}

\section{Proof of Theorem~\ref{THM31}: Tightness}\label{Se4}

Let $B^{(n)}$ be the rescaled processes described in Theorem~\ref{THM31}. In this section, we establish their tightness, i.e., the family $\{\bQ^{(n)}: n\geq 1\}$ of probability measures on $(D_{E_\Delta},\sD)$ defined as \eqref{eq:34} is tight. Concretely, for any $\varepsilon>0$, there exists a compact subset $K_\varepsilon$ of $D_{E_\Delta}$ with respect to the Skorohod topology such that $\inf_{n\geq 1}\bQ^{(n)}(K_\varepsilon)>1-\varepsilon$. 
A detailed discussion of this notion of tightness can be found in \cite{B99}. 

\subsection{Strategy of the proof}

Before proceeding, we introduce some notations for the special Skorohod space of real-valued functions. 
 Let $D_{\bR}:=D_{\bR}([0,\infty))$ be the Skorohod space consisting of all real-valued c\`adl\`ag functions on $[0,\infty)$. To distinguish it from $\sD$, we denote by $\sD_{\mathbb{R}}$ the Borel $\sigma$-algebra on $D_{\mathbb{R}}$. Denote by $\Lambda$ the class of all strictly increasing, continuous maps from $[0,\infty)$ onto itself, namely, each $\lambda\in \Lambda$ satisfies $\lambda(0)=0$ and $\lim_{t\rightarrow \infty}\lambda(t)=\infty$.
For $\lambda\in \Lambda$, define $\gamma(\lambda):=\sup_{0\leq t<s}\left|\log \frac{\lambda(s)-\lambda(t)}{s-t} \right|$. Note that 
\begin{equation}\label{eq:513}
    d(w_1,w_2):=\inf_{\lambda\in \Lambda, \gamma(\lambda)<\infty}\left[\gamma(\lambda)\vee \int_0^\infty e^{-s}d(w_1,w_2;\lambda, s)\dd s\right],\quad w_1,w_2\in D_\bR,
\end{equation}
where $d(w_1,w_2;\lambda, s):=\sup_{t\geq 0}|w_1(t\wedge s)-w_2(\lambda(t)\wedge s)| \wedge 1$, is a complete metric on $D_\bR$, which induces the Skorohod topology; see, e.g., \cite[Theorem~5.6 of Chapter 3]{EK86}.   
 It should be also pointed out that $d(w_k,w)\rightarrow 0$ for $w_k,w\in D_\bR$ if and only if, for any $T>0$, there exist functions $\lambda_k\in \Lambda$ such that 
\[
    \lim_{k\rightarrow \infty}\|w_k\circ \lambda_k-w\|_{u,T}=0
\]
and
\begin{equation}\label{eq:55}
    \lim_{k\rightarrow \infty}\sup_{t\in [0,T]}\|\lambda_k(t)-t\|=0,
\end{equation}
where $\|w_k\circ \lambda_k-w\|_{u,T}:=\sup_{t\in [0,T]}|w_k(\lambda_k(t))-w(t)|$ denotes the uniform norm on $[0,T]$; see, e.g., \cite[Proposition 5.3 of Chapter 3]{EK86}. 
Further, let $C_\bR:=C([0,\infty))$ denote the space of all continuous real-valued functions on $[0,\infty)$, endowed with the locally uniform topology. That is, $w_k\rightarrow w$ in $C_\bR$ if and only if, for any $T>0$, $\lim_{k\rightarrow \infty}\|w_k-w\|_{u,T}=0$. %Note that $C_\bR\subset D_\bR$, and if $w\in C_\bR$, then $w_k\rightarrow w$ relative to the Skorohod topology on $D_\bR$ if and only if the locally uniform convergence holds; see, e.g., \cite[VI. Proposition 1.17]{JŠ03}.
 Let $D_\bR\times C_\bR$ be the product space of $D_\bR$ and $C_\bR$, endowed with the product topology, whose Borel $\sigma$-algebra is denoted by $\sD_\bR\times \mathscr C_\bR$. 

Recall that $\cD(\sL)$, defined as \eqref{eq:37}, is dense in $C_0(E)$ with respect to the uniform norm. Observe that $E_\Delta$ may be identified with $[0,\infty]$, and every $f\in C_0(E)$ can be regarded as an element of $C(E_\Delta)$ by setting $f(\Delta)=0$. The main strategy of our proof is to apply the following result, taken from \cite[Chapter 3, Theorem 9.1]{EK86}.

\begin{lemma}
 Assume that for any $f\in \cD(\sL)$,  the family $\{f(B^{(n)}): n\geq 1\}$ is tight on $(D_\bR,\sD_\bR)$.
Then $\{B^{(n)}: n\geq 1\}$ is tight on $(D_{E_\Delta},\sD)$. 
\end{lemma}
\begin{proof}
    The compact containment condition \cite[Chapter 3, (9.1)]{EK86} is trivially satisfied in our setting, since $E_\Delta$ is compact. Moreover,  since $\sD(\sL)$ is dense in $C_0(E)$, it follows that
    \[
        H:=\{f+c\cdot 1_{E_\Delta}: f\in \cD(\sL), c\in \bR\}
    \]
    is dense in $C(E_\Delta)$. By \cite[Chapter 3, Theorem 9.1]{EK86}, it suffices to show that the family $\{f(B^{(n)})+c: n\geq 1\}$ is tight on $(D_\bR,\sD_\bR)$ for all $f\in \cD(\sL)$ and $c\in \bR$. Clearly, this tightness is equivalent to that of $\{f(B^{(n)}): n\geq 1\}$ by applying, e.g., \cite[Chapter 3, Theorem~7.2]{EK86}. The latter tightness is precisely the condition required in the statement of the lemma.
\end{proof}

Hence, our task reduces to establishing the tightness of $f(B^{(n)})$ on $(D_\bR,\sD_\bR)$ for each $f\in \cD(\sL)$, which can be accomplished by verifying Aldous' tightness criterion.
%The basic strategy of our proof is to apply Aldous's criterion for tightness. To state this result, we need to prepare some notions. 
Recall that $(\sG^{(n)}_t)_{t\geq 0}$ given by \eqref{eq:35} is the natural filtration of $B^{(n)}$. A non-negative random variable $\tau$ is called a $\sG^{(n)}_t$-\emph{stopping time} provided that $\{\tau\leq t\}\in \sG^{(n)}_t$ for any $t\geq 0$. Note that if $\tau$ is a $\sG^{(n)}_t$-stopping time, then $\tau+\delta$ remains a $\sG^{(n)}_t$-stopping time for any constant $\delta\geq 0$. For any $T>0$, let $\Upsilon^{(n)}_T$ denote the family of all $\sG^{(n)}_t$-stopping times bounded by $T$. 
With these preparations, we can state the following well-known result.

\begin{theorem}[Aldous \cite{Al78}]\label{THM41}
Let $f\in \cD(\sL)$. 
If the following two conditions hold:
\begin{itemize}
    \item[(1)] For any $T\geq 0$,
    \begin{equation}\label{eq:41}
\lim_{\ell \rightarrow \infty}\limsup_{n\rightarrow \infty} \mathbb{P}(\sup_{0\leq t\leq T}|f(B^{(n)}_t)| > \ell)=0,
    \end{equation}
    \item[(2)] For any $\varepsilon>0$ and $T>0$, 
\begin{equation}\label{eq:42}
     \lim_{\delta \rightarrow 0} \limsup_{n \to \infty} \sup_{\tau\in \Upsilon^{(n)}_T} \mathbb{P}(|f(B^{(n)}_{\tau + \delta}) - f(B^{(n)}_{\tau})| > \varepsilon) = 0,
\end{equation}
\end{itemize}
then the family $\{f(B^{(n)}): n\geq 1\}$ is tight on $(D_\bR,\sD_\bR)$. 
\end{theorem}
\begin{proof}
Note that \eqref{eq:41} is exactly the condition \cite[(16.22)]{B99}. In addition, \eqref{eq:42} provides an alternative formulation of Condition $1^\circ$ on page 176 of \cite{B99}. By \cite[Theorem~16.10]{B99}, these two conditions together imply the tightness of $\{B^{(n)}: n\geq 1\}$. 
\end{proof}

Since $f\in \cD(\sL)$ is bounded, the first condition \eqref{eq:41} in Theorem~\ref{THM41} is automatically  satisfied. Therefore, to establish the tightness of $\{B^{(n)}: n\geq 1\}$, we need only to verify the condition \eqref{eq:42}, which will be addressed in the next subsection. 

\subsection{Proof of \eqref{eq:42}}\label{SEC42}

In this subsection, we take a function $f\in \cD(\sL)$ and aim to prove \eqref{eq:42}. 
\iffalse the special auxiliary function 
\[
   f(x):=x, \quad x\geq 0,
\] 
which satisfies all the conditions stated in Lemma~\ref{LM42}, according to the assumption \eqref{eq:36}. The specific form of this auxiliary function will be used only once, namely in \eqref{eq:414}.\fi
Fix $\varepsilon, T, \delta>0$, $n$ and $\tau\in \Upsilon^{(n)}_T$. Our goal is to estimate
\begin{equation}\label{eq:414}
\mathbb{P}\left(\left|f\left(B^{(n)}_{\tau + \delta}\right) - f\left(B^{(n)}_{\tau}\right)\right| > \varepsilon\right). 
\end{equation}
We first record several useful facts for later use. Note that the norms $\|f\|_u$, $\|f'\|_u$ and $\|f''\|_u$ are all finite; see Remark~\ref{RM24}.  

\begin{lemma}
Let $f_n(j):=f(j/n)$ for all $j\in \bN$. Then 
\begin{equation}\label{eq:49}
\left|(P^{(n)}-I)f_n(i)\right|\leq \frac{\|f''\|_u}{2n^2},\quad \forall i>0,
 \end{equation}
 and
  \begin{equation}\label{eq:410}
|(P^{(n)}-I)f_n(0)|\leq \frac{p_1\|f\|_u+p_2\|f'\|_{u}+(2\|f\|_u+\|f'\|_{u})\int_{(0,\infty)}(1\wedge x)\,p_4(\dd x)}{n^2p_3},
 \end{equation}
 where $P^{(n)}$ is the transition matrix of $X^{(n)}$.
\end{lemma}
\begin{proof}
A straightforward computation shows that 
 \begin{equation}\label{eq:43}
    (P^{(n)}-I)f_n(i)=\frac{1}{2}\left(f_n(i+1)+f(i-1)-2f_n(i)\right),\quad i\geq 1,
 \end{equation}
 and 
 \begin{equation}\label{eq:44}
      (P^{(n)}-I)f_n(0)=-\fp^{(n)}_\Delta\cdot f_n(0)+\sum_{j\in \bN}\fp^{(n)}_j\cdot (f_n(j)-f_n(0)).
 \end{equation} 
    For $i>0$, it follows from \eqref{eq:43} and Taylor's theorem that
 \[
(P^{(n)}-I)f_n(i)= \frac{1}{2}\left(f\left(\frac{i+1}{n}\right)+f\left(\frac{i-1}{n}\right)-2f\left(\frac{i}{n}\right)\right)=\frac{f''(\theta_1)+f''(\theta_2)}{4n^2}
 \]
 for some $\theta_1\in ((i-1)/n,i/n)$ and $\theta_2\in (i/n, (i+1)/n)$. This establishes \eqref{eq:49}.  On the other hand, we have
 \begin{equation}\label{eq:519-2}
    f_n(1)-f_n(0)=f(1/n)-f(0)=\frac{f'(\theta_3)}{n}
 \end{equation}
 for some $\theta_3\in (0,1/n)$, and for $j\geq 2$,
 \begin{equation}\label{eq:520-2}
|f_n(j)-f_n(0)|=|f(j/n)-f(0)|\leq (2\|f\|_u+\|f'\|_{u})\cdot (1\wedge \frac{j}{n}). 
 \end{equation}
 Then \eqref{eq:410} can be easily deduced from \eqref{eq:44} and the expression of $(\fp^{(n)}_j)_{j\in \bN_\Delta}$. 
\end{proof}

Consider the process
\begin{equation}\label{eq:415}
    f\left(B^{(n)}_t\right)=f\left(\frac{1}{n}X^{(n)}_{\lfloor n^2t\rfloor}\right),\quad t\geq 0. 
\end{equation}
Note that the jumping measure of $X^{(n)}$ satisfies the assumptions of Lemma~\ref{LM21}. Consequently, Lemma~\ref{LM21} applies to each $X^{(n)}$ with $f_n(j)=f(j/n)$ for all $j\in \bN$, yielding that
\begin{equation}\label{eq:416}
\begin{aligned}
&M^{(n)}_0:=0,\\
&M^{(n)}_m:=f_n(X^{(n)}_m)-f_n(X^{(n)}_0)-\sum_{k=0}^{m-1} (P^{(n)}-I)f_n(X^{(n)}_k),\quad m\geq 1
\end{aligned}
\end{equation}
is a square-integrable martingale. For convenience, set
\[
\sigma_1:=\lfloor n^2\tau\rfloor,\quad \sigma_2:=\lfloor n^2(\tau+\delta)\rfloor.
\]
 Then, combining \eqref{eq:415} and \eqref{eq:416}, we obtain
\[
f\left(B^{(n)}_{\tau + \delta}\right) - f\left(B^{(n)}_{\tau}\right)=M^{(n)}_{\sigma_2}-M^{(n)}_{\sigma_1}+\sum_{k=\sigma_1}^{\sigma_2-1}(P^{(n)}-I)f_n(X^{(n)}_k). 
\]
We aim to estimate
\begin{equation}\label{eq:417}
\begin{aligned}
\bE&\left[\left(f\left(B^{(n)}_{\tau + \delta}\right) - f\left(B^{(n)}_{\tau}\right)\right)^2 \right] \\
&\leq 2\bE\left[\left(M^{(n)}_{\sigma_2}-M^{(n)}_{\sigma_1}\right)^2\right]+2\bE\left[\left(\sum_{k=\sigma_1}^{\sigma_2-1}(P^{(n)}-I)f_n(X^{(n)}_k)\right)^2\right] \\
&=:2J_1+2J_2. 
\end{aligned}\end{equation}
The second term $J_2$ can be estimated as follows. By using \eqref{eq:49} and \eqref{eq:410}, we have
\[
J_2\leq \left(\left(\frac{\|f''\|_u}{2n^2}\right)^2+\left(\frac{c_1}{n^2p_3}\right)^2\right)\cdot \bE\left[(\sigma_2-\sigma_1)^2\right],
\]
where $c_1$ denotes the numerator on the right hand side of \eqref{eq:410}.
Since $|\sigma_2-\sigma_1|=\lfloor n^2(\tau+\delta)\rfloor-\lfloor n^2\tau\rfloor\leq n^2\delta+1$, it follows that
\begin{equation}\label{eq:421}
    J_2\leq c_2\cdot \frac{(n^2\delta+1)^2}{n^4}, 
\end{equation}
where $c_2:=\|f''\|^2_u/4+c_1^2/p_3^2$ is a constant independent of $\varepsilon, T,\delta, n$ and $\tau$. 

To estimate the first term $J_1$, we need the following fact in order to apply the strong Markov property of $X^{(n)}$. 

\begin{lemma}\label{LM43}
For any $\tau\in \Upsilon^{(n)}_T$, the random variable $\lfloor n^2\tau\rfloor$ is a stopping time with respect to the natural filtration of $X^{(n)}$. 
\end{lemma}
\begin{proof}
Let $(\sG^{X^{(n)}}_m)_{m\geq 0}$ denote the natural filtration of $X^{(n)}$. Recall that the natural filtration $(\sG^{(n)}_t)_{t\geq 0}$ of $B^{(n)}$ satisfies \eqref{eq:35}, with $(\sG^{X^{(n)}}_m)_{m\geq 0}$ in place of $(\sG^{X}_m)_{m\geq 0}$. For any $j\in \bN$, we have
\[
\left\{\lfloor n^2\tau\rfloor=j\right\}=\left\{j/n^2\leq \tau<(j+1)/n^2 \right\}=\bigcap_{s<\frac{j}{n^2}}\bigcup_{t<\frac{j+1}{n^2}}\{s<\tau\leq t\},
\]
where $s,t$ range over all rational numbers satisfying the requirements. It follows from \eqref{eq:35} that $\left\{\lfloor n^2\tau\rfloor=j\right\}\in \sG^{X^{(n)}}_j$. Therefore, $\lfloor n^2\tau\rfloor$ is a stopping time with respect to $\sG^{X^{(n)}}_m$. 
\end{proof}

Before proceeding, we clarify the notation needed for applying the strong Markov property. Under the setup of Theorem \ref{THM31}, the BRW $X^{(n)}$ is defined on the probability space $(\Omega,\sG,\mathbb P)$ and satisfies
\[
    \mathbb P\left(X^{(n)}_0=\lfloor n x_0\rfloor\right)=1.
\]
When invoking the Markov or strong Markov property, we shall write $\mathbb P_{\lfloor n x_0\rfloor}$ in place of $\mathbb P$ to emphasize that the BRW starts from the state $\lfloor n x_0\rfloor$. (To distinguish this from the law of the FBM, we use subscripts for the initial state of the BRW, whereas superscripts are reserved for the FBM.) More generally, for any other $i\in\mathbb N$, we may consider the BRW starting from $i$ with the same transition matrix $P^{(n)}$; its law is denoted by $\mathbb P_i$. Since the strong Markov property is used only in the estimate of $J_1$, we will continue to write  $\mathbb P_{\lfloor n x_0\rfloor}$ as $\bP$ elsewhere, in order to keep the notation concise.
We also introduce the shift operators 
\[
    \theta_k:\Omega\to\Omega, \quad k\in\mathbb N,
\] which satisfy
$X^{(n)}_m\circ \theta_k = X^{(n)}_{m+k}$ for all $m\ge 0$.
Strictly speaking, the shift operators depend on $n$; however, this dependence will be suppressed for simplicity.

We now continue to estimate the term $J_1$ appearing in \eqref{eq:417}. Observe that $\lfloor n^2\delta\rfloor\leq \sigma_2-\sigma_1\leq \lfloor n^2\delta\rfloor+1$. With convention $k_\delta:=\lfloor n^2\delta\rfloor$ for convenience, we have
\[
|M^{(n)}_{\sigma_2}-M^{(n)}_{\sigma_1}|\leq \left|\sum_{j=0}^{k_\delta}\left(M^{(n)}_{\sigma_1+j+1}-M^{(n)}_{\sigma_1+j}\right)\right|+\left|M^{(n)}_{\sigma_2+1}-M^{(n)}_{\sigma_2}\right|.
\]
This implies that
\begin{equation}\label{eq:419-2}
\begin{aligned}
J_1&\leq 2\bE\left[\left(\sum_{j=0}^{k_\delta}\left(M^{(n)}_{\sigma_1+j+1}-M^{(n)}_{\sigma_1+j}\right)\right)^2\right]+2\bE\left[\left(M^{(n)}_{\sigma_2+1}-M^{(n)}_{\sigma_2}\right)^2\right]\\
&=:2J_{11}+2J_{12}.
\end{aligned}\end{equation}

Regarding the simpler term $J_{12}$, we note that $\tau+\delta\in \Upsilon^{(n)}_{T+\delta}$. By Lemma~\ref{LM43}, $\sigma_2=\lfloor n^2(\tau+\delta)\rfloor$ is a stopping time with respect to the natural filtration of $X^{(n)}$. It then follows from \eqref{eq:416} and the strong Markov property of $X^{n}$ that
\begin{equation}\label{eq:420}
\begin{aligned}
J_{12}&=\bE\left[\left(f_n\left(X^{(n)}_{\sigma_2+1}\right)-f_n\left(X^{(n)}_{\sigma_2}\right)-(P^{(n)}-I)f_n(X^{(n)}_{\sigma_2})\right)^2\right] \\
&=\bE\left[\left(f_n\left(X^{(n)}_{1}\right)-f_n\left(X^{(n)}_{0}\right)-(P^{(n)}-I)f_n(X^{(n)}_{0})\right)^2\circ \theta_{\sigma_2}\right]\\
&=\bE \left[\bE_{X^{(n)}_{\sigma_2}}\left[\left(f_n\left(X^{(n)}_{1}\right)-f_n\left(X^{(n)}_{0}\right)-(P^{(n)}-I)f_n(X^{(n)}_{0})\right)^2\right]\right]. 
\end{aligned}\end{equation}
For convenience, define for all $i\in \bN$, 
\begin{equation}\label{eq:421-2}
\begin{aligned}
g(i):=\bE_{i}&\left[\left(f_n\left(X^{(n)}_{1}\right)-f_n\left(X^{(n)}_{0}\right)-(P^{(n)}-I)f_n(X^{(n)}_{0})\right)^2\right] \\
&\leq 2\bE_{i}\left[\left(f_n\left(X^{(n)}_{1}\right)-f_n\left(X^{(n)}_{0}\right)\right)^2+\left((P^{(n)}-I)f_n(X^{(n)}_{0})\right)^2\right].
\end{aligned}\end{equation}
When $i\neq 0$, it holds that
\[
g(i)\leq (f_n(i+1)-f_n(i))^2+(f_n(i-1)-f_n(i))^2+2\left((P^{(n)}-I)f_n(i)\right)^2.
\]
By Taylor's theorem together with \eqref{eq:49}, we obtain
\begin{equation}\label{eq:418}
g(i)\leq \frac{2\|f'\|^2_u}{n^2}+\frac{\|f''\|_u^2}{2n^4},\quad \forall i\neq 0. 
\end{equation}
For $i=0$, we have
\[
g(0)\leq 2\sum_{j\in \bN}(f_n(j)-f_n(0))^2\fp^{(n)}_j+2\left((P^{(n)}-I)f_n(0)\right)^2. 
\]
Using \eqref{eq:519-2}, \eqref{eq:520-2}, the expression of $\fp^{(n)}_j$ and \eqref{eq:410},  we obtain
\begin{equation}\label{eq:419}
\begin{aligned}
g(0)&\leq \frac{2\|f'\|_{u}^2 \cdot p_2}{n^3p_3}+\frac{2(\|f'\|_u+2\|f\|_u)^2}{n^2p_3}\sum_{j=2}^{n^2-1}\left(1\wedge \frac{j}{n}\right)^2p_4\left((j/n,(j+1)/n]\right)+\frac{2c^2_1}{n^4p_3^2} \\
&\leq \frac{2\|f'\|_u^2 \cdot p_2}{n^3p_3}+\frac{2(\|f'\|_u+2\|f\|_u)^2\cdot \int_{(0,\infty)}(1\wedge x)\,p_4(\dd x)}{n^2p_3}+\frac{2c_1^2}{n^4p_3^2},
\end{aligned}\end{equation}
where $c_1$ is the constant appearing in the numerator on the right hand side of \eqref{eq:410}. Applying the estimates \eqref{eq:418} and \eqref{eq:419} to \eqref{eq:420} yields 
\begin{equation}\label{eq:422}
    J_{12}=\bE g(X^{(n)}_{\sigma_1})\leq c_3\cdot \left(\frac{1}{n^2}+\frac{1}{n^3}+\frac{1}{n^4}\right),
\end{equation}
where 
\[
        c_3:=\max\left\{2\|f'\|^2_u,\|f''\|_u^2,\frac{2\|f'\|_u^2  \cdot p_2}{p_3},\frac{2(\|f'\|_u+2\|f\|_u)^2 \int_{(0,\infty)}(1\wedge x)\,p_4(\dd x)}{p_3},\frac{2c_1^2}{p_3^2}\right\}
\]
is a constant independent of $n$. 

We now turn to the estimate of the term $J_{11}$. By an argument analogous to that leading to \eqref{eq:420}, using the strong Markov property at time $\sigma_1$, we obtain
\[
J_{11}=\bE\left[\bE_{X^{(n)}_{\sigma_1}}\left[\left(\sum_{j=0}^{k_\delta}\left(M^{(n)}_{j+1}-M^{(n)}_{j}\right)\right)^2\right]\right].
\]
Define
\[
h(i):=\bE_{i}\left[\left(\sum_{j=0}^{k_\delta}\left(M^{(n)}_{j+1}-M^{(n)}_{j}\right)\right)^2\right],\quad \forall i\in \bN. 
\]
The key observation is that, for each $i\in\bN$, the process $M^{(n)}$ remains a squre-integrable martingale under the law $\bP_i$. Consequently, using the martingale property, we have
\begin{equation}\label{eq:424-2}
h(i)=\bE_{i}\left[\sum_{j=0}^{k_\delta}\left(M^{(n)}_{j+1}-M^{(n)}_{j}\right)^2\right]=\sum_{j=0}^{k_\delta}\bE_i\left[\left(M^{(n)}_{j+1}-M^{(n)}_{j}\right)^2\right].
\end{equation}
For each $0\leq j\leq k_\delta$, an argument similar to that used in the estimate of \eqref{eq:420}, with $j$ in place of $\sigma_2$, yields 
\[
    \bE_i\left[\left(M^{(n)}_{j+1}-M^{(n)}_{j}\right)^2\right]\leq c_3\cdot \left(\frac{1}{n^2}+\frac{1}{n^3}+\frac{1}{n^4}\right). 
\]
It follows that
\[
    h(i)\leq c_3(1+k_\delta)\cdot \left(\frac{1}{n^2}+\frac{1}{n^3}+\frac{1}{n^4}\right),\quad \forall i\in \bN,
\]
which further implies
\begin{equation}\label{eq:423}
    J_{11}=\bE \left[h(X^{(n)}_{\sigma_1})\right]\leq c_3(1+k_\delta)\cdot \left(\frac{1}{n^2}+\frac{1}{n^3}+\frac{1}{n^4}\right),
\end{equation}
where $c_3$ is a constant independent of $\varepsilon, T,\delta, n$ and $\tau$. 

Finally, by combining \eqref{eq:414}, \eqref{eq:417}, \eqref{eq:421}, \eqref{eq:422} and \eqref{eq:423}, and noting that $k_\delta=\lfloor n^2\delta\rfloor\leq n^2\delta$, we can conclude
\[
\bP(|B^{(n)}_{\tau + \delta} - B^{(n)}_{\tau}| > \varepsilon)\leq \frac{2}{\varepsilon^2}\left[c_2\cdot \frac{(n^2\delta+1)^2}{n^4}+c_3(2+n^2\delta)\cdot \left(\frac{1}{n^2}+\frac{1}{n^3}+\frac{1}{n^4}\right)\right],
\]
where neither $c_2$ nor $c_3$ depends on $\varepsilon, T,\delta, n$ or $\tau$. This readily leads to \eqref{eq:42}. Eventually, the tightness of $\{B^{(n)}:n\geq 1\}$ is established.

\section{Proof of Theorem~\ref{THM31}: Weak convergence}\label{SEC5}

Recall that the rescaled BRW $B^{(n)}$ described in Theorem~\ref{THM31} can be realized as a law $\bQ^{(n)}$ on the Skorohod space as in \eqref{eq:34}, and that the FBM $B$ starting from $x_0$ in Theorem~\ref{THM31} can likewise be  realized as a law $\bQ^{x_0}$ on the same Skorohod space. The generator $\sL$ with domain $\cD(\sL)$ is clarified in Theorem~\ref{THM23}. 
Having established the tightness of $\bQ^{(n)}$, we now appeal to the uniqueness of martingale solutions associated with the Feller process to conclude that $\bQ^{(n)}$ converges weakly to $\bQ^{x_0}$. This result is formulated precisely in the following lemma.

\begin{lemma}\label{LM51}
Let $\sQ$ denote the collection of all probability measures $\bQ$ on $(D_{E_\Delta}, \sD)$ such that there exists a subsequence \(\{\bQ^{(n_k)}\}\) of $\{\bQ^{(n)}\}$ converging weakly to \(\bQ\).
Suppose that every $\bQ\in\sQ$ satisfies 
\begin{equation}\label{eq:51}
    \bQ(\{w\in D_{E_\Delta}: w(0)=x_0\})=1,
\end{equation}
and that, for every $f\in\cD(\sL)$,
\begin{equation}\label{eq:52}
N_t^f := f(w(t)) - f(w(0)) - \int_0^t \sL f(w(s))\,\mathrm d s,\qquad t\ge 0,
\end{equation}
is a martingale on $(D_{E_\Delta}, \sD, \bQ)$ with respect to the filtration $(\sD_t)_{t\ge 0}$. Then $\sQ$ contains exactly one element, namely $\bQ^{x_0}$. In particular, $\bQ^{(n)} \Rightarrow \bQ^{x_0}$.
\end{lemma}
\begin{proof}
See, e.g., \cite[Theorem 3.33]{Lig10}. 
\end{proof}

Fix $\bQ\in \sQ$, and assume, without loss of generality, that $\bQ^{(n)}\Rightarrow \bQ$. The objective of this section is to show that $\bQ$ satisfies the two conditions stated in Lemma~\ref{LM51}. 

The first condition \eqref{eq:51} is straightforward to verify. Indeed, since the projection map $\pi_0$ is continuous on $D_{E_\Delta}$ (see \cite[Theorem 12.5]{B99}), the continuous mapping theorem (see \cite[Theorem 2.7]{B99}) yields that
\[
\bQ^{(n)}\circ \pi_0^{-1}=\bP\circ \left(B^{(n)}_0\right)^{-1}
=\delta_{\frac{\lfloor nx_0\rfloor}{n}}
\]
converges weakly to $\bQ\circ \pi_0^{-1}$. As a result, $\bQ\circ \pi_0^{-1}=\delta_{x_0}$, and \eqref{eq:51} follows immediately.

From now on, fix $f\in \cD(\sL)$, and we turn to prove that \eqref{eq:52} is a martingale on $(D_{E_\Delta}, \sD, \bQ)$. To proceed, we first introduce some notations. Define
\begin{equation}\label{eq:58}
\left(\sL^{(n)}f\right)\left(x\right):=n^2\cdot (P^{(n)}-I)f_n(i),\quad \forall x\in [i/n, (i+1)/n),\; i\in \bN,
\end{equation}
where $f_n(i):=f(i/n)$ for $i\in \bN$. %This is exactly the discrete generator of discrete-time Markov chain $\left\{\frac{1}{n}X^{(n)}_{n^2t}:t=\frac{k}{n^2}, k\in \bN \right\}$
For each $w\in D_{E_\Delta}$, set
\begin{equation}\label{eq:54}
N^{(n)}_t(w):=f(w(t))-f(w(0))-\int_0^{\frac{\lfloor n^2t\rfloor}{n^2}}\left(\sL^{(n)}f\right)(w(s))\dd s,\quad t\geq 0. 
\end{equation}
Then $(N^{(n)}_t)_{t\geq 0}$ forms a $\sD_t$-adapted c\`adl\`ag real-valued process defined on the probability space $(D_{E_\Delta},\sD,\bQ^{(n)})$, which is equivalent to the real-valued c\`adl\`ag process
\begin{equation}\label{eq:53}
f(B^{(n)}_t)-f(B^{(n)}_0)-\int_0^{\frac{\lfloor n^2t\rfloor}{n^2}}\left(\sL^{(n)}f\right)(B^{(n)}_s)\dd s,\quad t\geq 0,
\end{equation}
on the probability space $(\Omega,\sG,\bP)$. 
 
The key step is to establish the following weak convergence.

\begin{proposition}\label{PRO52}
Let $f\in \cD(\sL)$, and let $N^{(n)}$ be the c\`adl\`ag process on the probability space $(D_{E_\Delta},\sD,\bQ^{(n)})$ defined as \eqref{eq:54}. Further let $N^f$ be the c\`adl\`ag process on $(D_{E_\Delta},\sD,\bQ)$ defined as \eqref{eq:52}. Assume that $\bQ^{(n)}\Rightarrow \bQ$ as $n\rightarrow \infty$.  Then $N^{(n)}\Rightarrow N^f$ as $n\rightarrow \infty$.
\end{proposition}
\begin{proof}
Let $(M^{(n)}_t)_{t\geq 0}$ denote the process defined as \eqref{eq:53}. As explained in Remark~\ref{RM31}, it suffices to establish $M^{(n)}\Rightarrow N^f$. To this end, we decompose $M^{(n)}$ as $M^{(n)}_t=M^{(n),1}_t+M^{(n),2}_t$, where
\[
M^{(n),1}_t:=f(B^{(n)}_t)-f(B^{(n)}_0)-\int_0^{t}\left(\sL f\right)(B^{(n)}_s)\dd s
\]
and
\[
  M^{(n),2}_t:=  \int_0^{t}\left(\sL f\right)(B^{(n)}_s)\dd s-\int_0^{\frac{\lfloor n^2t\rfloor}{n^2}}\left(\sL^{(n)}f\right)(B^{(n)}_s)\dd s
\]
are both c\`adl\`ag real-valued processes. 

We first show that $M^{(n),1}\Rightarrow N^f$.
Consider the following maps:
\[
\begin{aligned}
&\phi_1: D_{E_\Delta}\rightarrow D_\bR,\quad w\mapsto f(w(\cdot))-f(w(0)),\\
&\phi_2: D_{E_\Delta}\rightarrow C_\bR,\quad w\mapsto \int_0^\cdot \sL f(w(s))\dd s,
\end{aligned}\]
and
\[
\begin{aligned}
    &\Phi: D_{E_\Delta}\rightarrow D_\bR\times C_\bR,\quad w\mapsto (\phi_1(w), \phi_2(w)), \\
    &\Psi: D_\bR\times C_\bR\rightarrow D_\bR,\quad (w_1,w_2)\mapsto w_1-w_2,
\end{aligned}\]
where $(w_1-w_2)(t):=w_1(t)-w_2(t)$ for $t\geq 0$. 
We claim that all these maps are continuous, and hence so is
\begin{equation}\label{eq:57}
    \Psi\circ \Phi:D_{E_\Delta}\rightarrow D_\bR. 
\end{equation}
The continuity of $\Psi$ has been stated in, e.g., \cite[VI. Proposition 1.23]{JS03}. We next verify that continuity of $\phi_2$. Let $w_k\rightarrow w$ in $D_{E_\Delta}$ with respect to the Skorohod topology, and fix $T>0$. It is known that for a.e. $t\in [0,T]$, $\lim_{k\rightarrow \infty}w_k(t)=w(t)$ holds; see, e.g., \cite[Proposition~5.2 of Chapter 3]{EK86}. Since $\sL f\in C_0(E)$ is bounded, it follows from the dominated convergence theorem that
\[
   \lim_{k\rightarrow \infty}\left\|\phi_2(w_k)-\phi_2(w)\right\|_{u,T}\leq \int_0^T \lim_{k\rightarrow \infty}\left| \sL f(w_k(s))-\sL f(w(s))\right| \dd s=0.
\]
This establishes the continuity of $\phi_2$. It remains to verify the continuity of $\phi_1$. Recall that $E_\Delta$ may be identified with $[0,\infty]$, equipped with the metric $r(x,y):=|\arctan x-\arctan y|$. Let $w_k\rightarrow w$ in $D_{E_\Delta}$ and fix $T>0$. According to \cite[Proposition~5.3 of Chapter 3]{EK86}, there exist $\lambda_k\in \Lambda$ such that \eqref{eq:55} holds and 
\begin{equation}\label{eq:56}
    \lim_{k\rightarrow \infty} \sup_{t\in [0,T]}r(w_k(\lambda_k(t)), w(t))=0. 
\end{equation}
We then have
\[
\|\phi_1(w_k\circ \lambda_k)-\phi_1(w)\|_{u,T}\leq |f(w_k(0))-f(w(0))|+ \sup_{t\in [0,T]}|f(w_k(\lambda_k(t)))-f(w(t))|.
\]
Since $f\in C_0(E)$ can be viewed as a continuous function on $[0,\infty]$ by setting $f(\infty)=0$, which is uniformly continuous on $[0,\infty]$ with respect to the metric $r$, and $w_k(0)\rightarrow w(0)$, it follows from \eqref{eq:56} that $\|\phi_1(w_k\circ \lambda_k)-\phi_1(w)\|_{u,T}\rightarrow 0$. Therefore, $d(\phi_1(w_k), \phi_1(w))\rightarrow 0$, proving the continuity of $\phi_1$. 

Having established the continuity of \eqref{eq:57}, we note that $M^{(n),1} = \Psi\circ \Phi(B^{(n)})$, while $N^f = \Psi\circ \Phi(\pi)$, where $\pi = (\pi_t)_{t\ge 0}$ denotes the coordinate process on $(D_{E_\Delta}, \mathscr{D}, \mathbb{Q})$. Since $\bQ^{(n)}\Rightarrow \bQ$ implies $B^{(n)}\Rightarrow \pi$, the continuous mapping theorem yields $M^{(n),1}\Rightarrow N^f$. 
\iffalse
Note that both $\phi_1$ and $\phi_2$ are $\sD/\sD_\bR$-measurable, which indicates that $\Phi$ is $\sD/(\sD_\bR\times \sD_\bR)$-measurable. Indeed, since $\sD_\bR$ is generated by all projections 
\[
    \pi^\bR_t:D_\bR\rightarrow \bR,\quad w\mapsto w(t),
\] 
it suffices to show that for any $t\feq 0$, both $\pi^\bR_t\circ \phi_1$ and $\pi^\bR_t\circ \phi_2$ are $\sD$-measurable. The $\sD$-measurability of $\pi^\bR_t\circ \phi_1$ is clear, because $\pi^\bR_t\circ \phi_1=f\circ \pi_t-f\circ \pi_0$. The $\sD$-measurability of $\pi^\bR_t\circ \phi_2$ can be deduced as follows: Since the bounded c\`adl\`ag function $s\mapsto \sL f(w(s))$ is Riemann integrable on $[0,t]$, it follows that 
\[
\pi^\bR_t\circ \phi_2(w)= \lim_{m\rightarrow \infty}\frac{1}{m}\sum_{j=0}^{\lfloor mt\rfloor}\sL f(w(j/m))=\lim_{m\rightarrow \infty}\frac{1}{m}\sum_{j=0}^{\lfloor mt\rfloor}(\sL f)\circ \pi_{j/m}(w),
\]
which readily implies the $\sD$-measurability of $\pi^\bR_t\circ\phi_2$. \fi

Next, we show that for any $T>0$, 
\begin{equation}\label{eq:512}
\lim_{n\rightarrow \infty}\sup_{t\in [0,T]}\left|M^{(n),2}_t\right|=0,\quad \bP\text{-a.s.}
\end{equation}
A straightforward calculation yields
\begin{equation}\label{eq:59}
    \sup_{t\in [0,T]}\left|M^{(n),2}_t\right|\leq \int_0^T\left| \left(\sL f-\sL^{(n)}f\right)(B^{(n)}(s))\right|\dd s+\frac{1}{n^2}\|\sL f\|_u. 
\end{equation}
For $x\neq 0$, it follows from \eqref{eq:58} and Taylor's theorem that
\[
\begin{aligned}
    \sL f(x)-\sL^{(n)}f(x)&=\frac{1}{2}f''(x)-\frac{n^2}{2}\left(f\left(\frac{\lfloor nx\rfloor+1}{n}\right)+f\left(\frac{\lfloor nx\rfloor -1}{n}\right)-2f\left(\frac{\lfloor nx\rfloor}{n}\right)\right) \\
    &=\frac{1}{2}f''(x)-\frac{f''(\theta_1)+f''(\theta_2)}{4}
\end{aligned}
\]
for some $\theta_1\in (\frac{\lfloor nx\rfloor -1}{n}, \frac{\lfloor nx\rfloor}{n})$ and $\theta_2\in (\frac{\lfloor nx\rfloor}{n}, \frac{\lfloor nx\rfloor+1}{n})$.  Let $\delta_{f''}$ denote the modulus of continuity of $f''$, namely,
\[
    \delta_{f''}(t):=\sup_{x,y\in E: |x-y|<t}|f''(x)-f''(y)|,\quad t>0. 
\]
Since $f''\in C_0(E)$ is uniformly continuous, we have $\lim_{t\rightarrow 0}\delta_{f''}(t)=0$. Noting $|x-\theta_1|,|x-\theta_2|<1/n$, we obtain
\begin{equation}\label{eq:510}
    |\sL^{(n)}f(x)-\sL f(x)|\leq \frac{\delta_{f''}(1/n)}{2},\quad x\neq 0. 
\end{equation}
For $x=0$, we have
\begin{equation}\label{eq:515}
\begin{aligned}
\sL^{(n)}f(0)&=n^2(P^{(n)}-I)f_n(0)=n^2\left(\fp^{(n)}_\Delta (-f(0))+\sum_{j\in \bN}\fp^{(n)}_j (f(j/n)-f(0)) \right) \\
&=\frac{1}{p_3}\left(-p_1f(0)+p_2f'(\theta^{(n)}_3)+\int_{(0,\infty)}h_n(x)\,p_4(\dd x) \right)
\end{aligned}\end{equation}
for some $\theta^{(n)}_3\in (0,1/n)$, where $h_n(x):=\sum_{j= 2}^{n^2-1}(f(j/n)-f(0))1_{(j/n,(j+1)/n]}(x)$. Since $f\in \cD(\sL)$ satisfies the boundary condition stated in \eqref{eq:37}, it follows that
\[
\begin{aligned}
p_3(\sL f(0)-\sL^{(n)}f(0))&=p_2(f'(0)-f'(\theta^{(n)}_3))+\int_{(0,\infty)}(f(x)-f(0))\,p_4(\dd x)\\
&\qquad -\int_{(0,\infty)}h_n(x)\,p_4(\dd x).
\end{aligned}\]
Moreover,
\begin{equation}\label{eq:516}
    |h_n(x)|\leq \left(\|f'\|_u+2\|f\|_u\right)\cdot (1\wedge x)\in L^1((0,\infty),p_4)
\end{equation}
and $\lim_{n\rightarrow \infty}h_n(x)=f(x)-f(0)$ for all $x\in (0,\infty)$. By the continuity of $f'$ and the dominated convergence theorem, we conclude that
\begin{equation}\label{eq:511}
    \chi(n):=\sL f(0)-\sL^{(n)}f(0)\rightarrow 0
\end{equation}
as $n\rightarrow \infty$. Combining \eqref{eq:59}, \eqref{eq:510} and \eqref{eq:511}, we obtain
\[
    \sup_{t\in [0,T]}\left|M^{(n),2}_t\right|\leq \left(\frac{\delta_{f''}(1/n)}{2}+\chi(n)\right)\cdot T +\frac{1}{n^2}\|\sL f\|_u\rightarrow 0,
\]
as $n\rightarrow \infty$, which readily establishes \eqref{eq:512}.

Finally, we assert that
\begin{equation}\label{eq:514}
   \lim_{n\rightarrow \infty} d(M^{(n)},M^{(n),1})=0,\quad \bP\text{-a.s.},
\end{equation}
where $d$ is the metric on $D_\bR$ given by \eqref{eq:513}. Then, since $D_\bR$ is separable (see, e.g., \cite[Theorem~5.6 of Chapter 3]{EK86}), it follows from \cite[Corollary~3.3 of Chapter 3]{EK86} that \eqref{eq:514}, together with the weak convergence $M^{(n),1}\Rightarrow N^f$, eventually yields the desired conclusion $M^{(n)}\Rightarrow N^f$. 

Indeed, noting that $\lambda_0\in \Lambda$ and $\gamma(\lambda_0)=0$ for $\lambda_0(t):=t$, $t\geq 0$, we obtain that for any $T>0$, 
\[
\begin{aligned}
d(M^{(n)},M^{(n),1})&\leq \int_0^T e^{-u} \left(\sup_{t\in [0,u]} |M^{(n)}_t-M^{(n),1}_t|\wedge 1\right)\dd u+\int_T^\infty e^{-u}\dd u \\
&\leq \sup_{t\in [0,T]}|M^{(n),2}_t|+e^{-T}. 
\end{aligned}\]
It then follows from \eqref{eq:512} that
\[
    \limsup_{n\rightarrow \infty} d(M^{(n)},M^{(n),1})\leq e^{-T},\quad \bP\text{-a.s.}
\]
Letting $T\uparrow \infty$ yields \eqref{eq:514}. This completes the proof. 
\end{proof}

We are now in a position to complete the proof of  Theorem~\ref{THM31}.

\begin{proof}[Proof of Theorem~\ref{THM31}]
In light of Lemma~\ref{LM51}, it suffices to show that $N^f$ defined as \eqref{eq:52} is a martingale on $(D_{E_\Delta}, \sD, \bQ)$. To this end, we first observe that for $t\geq 0$, \eqref{eq:53} is equal to
\[
f_n\left(X^{(n)}_{\lfloor n^2t\rfloor}\right)-f_n\left(X^{(n)}_0\right)-\sum_{j=0}^{\lfloor n^2t\rfloor-1}(P^{(n)}-I)f_n\left(X^{(n)}_j\right). 
\]
According to Lemma~\ref{LM21}, this process is a martingale on $(\Omega, \sG,\bP)$ adapted to the natural filtration $(\sG^{(n)}_t)_{t\geq 0}$ of $B^{(n)}$. Consequently, $N^{(n)}$ defined as \eqref{eq:54} is a martingale on $(D_{E_\Delta}, \sD, \bQ^{(n)})$. For any $n$, $t> 0$ and $w\in D_{E_\Delta}$, we have
\[
    |N^{(n)}_t(w)-N^{(n)}_{t-}(w)|\leq 2\|f\|_u+\frac{1}{n^2}\left|\sL^{(n)}f(w(\lfloor n^2t\rfloor/n^2))\right|,
\]
where $N^{(n)}_{t-}:=\lim_{s\uparrow t}N^{(n)}_s$. 
For $x\neq 0$, it is immediate that
\[
    \frac{1}{n^2}\left|\sL^{(n)}f(x)\right| \leq 2\|f\|_u. 
\]
From \eqref{eq:515} and \eqref{eq:516}, we obtain that
\begin{equation}\label{eq:517}
    \frac{1}{n^2}\left|\sL^{(n)}f(0)\right|\leq \frac{c}{n^2},
\end{equation}
where $c:=\left(p_1\|f\|_u+p_2\|f'\|_{u}+\left(\|f'\|_{u}+2\|f\|_u\right)\cdot\int_{(0,\infty)}(1\wedge x)\,p_4(\dd x)\right)/p_3$ is a constant independent of $n, t$ or $w$. As a result,
\[
    |N^{(n)}_t(w)-N^{(n)}_{t-}(w)|\leq 4\|f\|_u+c,\quad \forall t> 0, n\geq 1, w\in D_{E_\Delta}. 
\]
Note that $N^{(n)}\Rightarrow N^f$ has been established in Proposition~\ref{PRO52}. Applying \cite[IX., Corollary 1.19]{JS03} to $N^{(n)}$ and $N^f$, we can conclude that $N^f$ is a local martingale on $(D_{E_\Delta},\sD,\bQ)$. Finally, for any $T>0$,
\[
|N^f_t|\leq 2\|f\|_u+\frac{1}{2}\|f''\|_u\cdot T,\quad \forall 0\leq t\leq T, 
\]
so that $N^f$ is in fact a martingale. (It is well known that any bounded local martingale is a true martingale.) This completes the proof. 
\end{proof}

\section{Proof of Theorem~\ref{THM35}}\label{SEC7}

The proof of Theorem~\ref{THM35} follows the same overall strategy as that of Theorem~\ref{THM31}. However, a fundamental difficulty arises in the present setting. For the BRW $X^{(n)}$ corresponding to \eqref{eq:316} and \eqref{eq:317}, the rough estimate that was repeatedly employed in Sections \ref{Se4} and \ref{SEC5},
\[
\mathbb{E}\!\left[\sum_{k=0}^{\lfloor n^{2}t\rfloor-1} \mathbf{1}_{\{X^{(n)}_k=0\}}\right]\le n^{2}t,\qquad \forall t>0,
\]
is no longer sufficient for our purposes. To overcome this difficulty, we appeal to Lemma~\ref{LM22}, which provides the finer estimates required in this case.

\begin{lemma}\label{LM81}
Let $X^{(n)}$ be a BRW with jumping measure given by \eqref{eq:316} and \eqref{eq:317},  with no restriction on the initial state.
For any positive integer $m$, the following estimate holds:
\begin{equation}\label{eq:71}
   \mathbb{E}\!\left[\sum_{k=0}^{m-1} \mathbf{1}_{\{X^{(n)}_k=0\}}\right]\le \frac{e}{\sqrt{1-e^{-\frac{2}{m}}}}. 
\end{equation}
\end{lemma}
\begin{proof}
    We adopt the same notation $\bP_i$ for $i\in \bN$, with the corresponding expectation denoted by $\bE_i$,  as in Section~\ref{SEC42} to indicate that $X^{(n)}$ starts from $i$. It suffices to prove \eqref{eq:71} for each $\bE_i$ in place of $\bE$. The case $i=0$ follows directly from \eqref{eq:25-2}. The general case can be handled by an application of the strong Markov property of $X^{(n)}$. To this end, define $\tau:=\inf\{k\geq 0: X^{(n)}_k=0\}$.
    We then have
    \[
    \begin{aligned}
    \mathbb{E}_i\!\left[\sum_{k=0}^{m-1} \mathbf{1}_{\{X^{(n)}_k=0\}}\right]&=\mathbb{E}_i\!\left[\sum_{k=\tau}^{m-1} \mathbf{1}_{\{X^{(n)}_k=0\}}; \tau\leq m-1\right] \\
    &\leq \mathbb{E}_i\!\left[\left(\sum_{k=0}^{m-1} \mathbf{1}_{\{X^{(n)}_k=0\}}\right)\circ \theta_\tau; \tau\leq m-1\right]
    \end{aligned}\]
    Applying the strong Markov property of $X^{(n)}$ yields that 
    \[
    \begin{aligned}
       \mathbb{E}_i\!\left[\sum_{k=0}^{m-1} \mathbf{1}_{\{X^{(n)}_k=0\}}\right]&\leq \bE_i\left[\bE_{X^{(n)}_\tau}\left[\sum_{k=0}^{m-1} \mathbf{1}_{\{X^{(n)}_k=0\}} \right]; \tau\leq m-1 \right] \\
      & =\bE_i\left[\bE_{0}\left[\sum_{k=0}^{m-1} \mathbf{1}_{\{X^{(n)}_k=0\}} \right]; \tau\leq m-1 \right]\\
      &\le \frac{e}{\sqrt{1-e^{-\frac{2}{m}}}}. 
    \end{aligned}\]
    This completes the proof.
\end{proof}
\begin{remark}
It is worth noting that 
\begin{equation}\label{eq:73}
    \lim_{m\rightarrow \infty} m\cdot \left(1-e^{-\frac{2}{m}}\right)=2. 
\end{equation}
As a consequence, the quantity on the right hand side of \eqref{eq:71} grows asymptotically like $e\cdot \sqrt{m/2}$ as $m\rightarrow \infty$. 
\end{remark}

\subsection{Proof of tightness}

To obtain the tightness of $B^{(n)}$, it suffices to prove the condition \eqref{eq:42} for each $f\in \cD(\sL)$ as stated in Theorem~\ref{THM41}. We retain the notation and follow the same scheme as in Section~\ref{SEC42}.

%We adopt the notation of Section~\ref{SEC41} and proceed along the same lines to establish \eqref{eq:41}.

\iffalse
 \begin{proof}[Proof of \eqref{eq:41}]
Let $t>0$ be fixed.  In the present setting, the preceding discussion leading up to \eqref{eq:410} remains applicable, whereas \eqref{eq:410} itself is replaced by the following estimate: 
\begin{equation}\label{eq:72}
|(P^{(n)}-I)f_n(0)|\leq \frac{\|f'\|_u+\int_{(0,\infty)}f(x)\,p_4(\dd x)/p_2}{n}.
 \end{equation}
 Then the estimates \eqref{eq:49}, \eqref{eq:71} and \eqref{eq:72} imply 
 \[
     \left|\bE\left[\sum_{k=0}^{m-1}h(X^{(n)}_k)\right]\right|\leq \frac{t\cdot \|f''\|_u}{2}+\frac{\|f'\|_u+\int_{(0,\infty)}f(x)\,p_4(\dd x)/p_2}{n}\frac{e}{\sqrt{1-e^{-\frac{2}{m}}}}
 \]
where $h:=(P^{(n)}-I)f_n$ and $m:=\lfloor n^2t\rfloor$. By noting that
\begin{equation}\label{eq:73}
    \lim_{n\rightarrow \infty} \lfloor n^2t\rfloor\cdot \left(1-e^{-\frac{2}{\lfloor n^2t\rfloor}}\right)=2, 
\end{equation}
we can deduce that
\[
\limsup_{n\rightarrow\infty}\bE\left[f\left(\frac{1}{n}X^{(n)}_{\lfloor n^2t\rfloor} \right) \right]\leq f(x_0)+\frac{t\cdot \|f''\|_u}{2}+e\sqrt{t/2}\left(\|f'\|_u+\frac{\int_{(0,\infty)}f(x)\,p_4(\dd x)}{p_2}\right). 
\]
Therefore, \eqref{eq:41} follows from \eqref{eq:412} and  $\lim_{\ell\uparrow\infty} f(\ell)=\infty$.
 \end{proof} 
The proof of \eqref{eq:42} is slightly more delicate.  
\fi

\begin{proof}[Proof of \eqref{eq:42}]
Note that \eqref{eq:49} remains applicable, whereas in the present setting, \eqref{eq:410} should be replaced by the following estimate: 
\begin{equation}\label{eq:72}
|(P^{(n)}-I)f_n(0)|\leq \frac{p_1\|f\|_u+p_2\|f'\|_u+(2\|f\|_u+\|f'\|_u)\int_{(0,\infty)}(1\wedge x)\,p_4(\dd x)}{np_2}.
 \end{equation}
In contrast with \eqref{eq:417}, the present argument requires an estimate of
\[
\begin{aligned}
\bE&\left[\left|f\left(B^{(n)}_{\tau + \delta}\right) - f\left(B^{(n)}_{\tau}\right)\right|\right] \\
&\leq \sqrt{\bE\left[\left(M^{(n)}_{\sigma_2}-M^{(n)}_{\sigma_1}\right)^2\right]}+\bE\left[\left|\sum_{k=\sigma_1}^{\sigma_2-1}(P^{(n)}-I)f_n(X^{(n)}_k)\right|\right] \\
&=:\sqrt{J_1}+J'_2. 
\end{aligned}
\]
Using \eqref{eq:49} and \eqref{eq:72}, and noting that $|\sigma_2-\sigma_1|\leq n^2\delta +1$, we obtain
\[
    J'_2\leq \frac{\|f''\|_u(n^2\delta+1)}{2n^2}+\frac{c'_1}{np_2}\bE\left[\sum_{k=\sigma_1}^{\sigma_2-1}1_{\{X^{(n)}_k=0\}} \right],
\]
where $c'_1$ is the numerator on the right hand side of \eqref{eq:72}. 
It follows from the strong Markov property of $X^{(n)}$ and \eqref{eq:71} that 
\[
\begin{aligned}
\bE\left[\sum_{k=\sigma_1}^{\sigma_2-1}1_{\{X^{(n)}_k=0\}} \right]&\leq \bE \left[\left(\sum_{k=0}^{\lfloor n^2\delta\rfloor+1}1_{\{X^{(n)}_k=0\}}\right)\circ \theta_{\sigma_1}\right] \\
&=\bE \left[\bE_{X_{\sigma_1}}\left[\sum_{k=0}^{\lfloor n^2\delta\rfloor+1}1_{\{X^{(n)}_k=0\}} \right] \right] \\
&\le \frac{e}{\sqrt{1-e^{-\frac{2}{\lfloor n^2\delta\rfloor+2}}}}. 
\end{aligned}\]
This estimate, together with \eqref{eq:73}, yields
\[
    \limsup_{n\rightarrow \infty}\sup_{\tau\in \Upsilon^{(n)}_T}J'_2\leq \frac{\|f''\|_u}{2}\cdot \delta+e\sqrt{\delta/2}\cdot \frac{c'_1}{p_2}. 
\]
Therefore, 
\begin{equation}\label{eq:74}
\lim_{\delta\rightarrow 0}\limsup_{n\rightarrow \infty}\sup_{\tau\in \Upsilon^{(n)}_T}J'_2=0.
\end{equation}

To examine the term $J_1$, we observe that \eqref{eq:419-2} remains valid. In addition, the preceding discussion leading up to \eqref{eq:419} continues to apply, while \eqref{eq:419} should be replaced by
\begin{equation}\label{eq:75}
g(0)\leq \frac{2\|f'\|_u^2}{n^2}+\frac{2(\|f'\|_u+2\|f\|_u)^2\int_{(0,\infty)}(1\wedge x)\,p_4(\dd x)}{np_2}+2\left(\frac{c'_1}{n}\right)^2=:\beta(n).
\end{equation}
%where $c'_1:=\|f'\|_u+\int_{(0,\infty)}f(x)\,p_4(\dd x)/p_2$. 
This estimate, together with \eqref{eq:418}, yields that
\begin{equation}\label{eq:76}
    \limsup_{n\rightarrow \infty}\sup_{\tau\in \Upsilon^{(n)}_T}J_{12}=0. 
\end{equation}
We next consider the remaining term $J_{11}$. Note that
\[
    J_{11}=\bE\left[h(X^{(n)}_{\sigma_1})\right],
\]
where $h$ is given by \eqref{eq:424-2}. For $i\in \bN$ and each $0\leq j\leq k_\delta$, the Markov proprety of $X^{(n)}$ implies that
\[
\bE_i\left[\left(M^{(n)}_{j+1}-M^{(n)}_j\right)^2 \right]=\bE_i\left[g(X^{(n)}_j) \right],
\]
where $g$ is given in \eqref{eq:421-2}. Therefore, using \eqref{eq:418}, \eqref{eq:75} and \eqref{eq:71}, we obtain
\[
\begin{aligned}
h(i)&\leq \left(\frac{2\|f'\|^2_u}{n^2}+\frac{\|f''\|_u^2}{2n^4}\right)\cdot n^2\delta+\beta(n)\cdot \bE_i\left[ \sum_{j=0}^{k_\delta}1_{\{X^{(n)}_j=0\}}\right] \\
&\leq \left(\frac{2\|f'\|^2_u}{n^2}+\frac{\|f''\|_u^2}{2n^4}\right)\cdot n^2\delta+\beta(n)\cdot \frac{e}{\sqrt{1-e^{-\frac{2}{k_\delta+1}}}},
\end{aligned}\]
where $\beta(n)$ is given in \eqref{eq:75}. Consequently, 
\[
    J_{11}\leq \left(\frac{2\|f'\|^2_u}{n^2}+\frac{\|f''\|_u^2}{2n^4}\right)\cdot n^2\delta+\beta(n)\cdot \frac{e}{\sqrt{1-e^{-\frac{2}{k_\delta+1}}}}. 
\]
By utilizing \eqref{eq:73} with $m=k_\delta+1$, we conclude that
\begin{equation}\label{eq:77}
     \lim_{\delta\rightarrow 0}\limsup_{n\rightarrow \infty}\sup_{\tau\in \Upsilon^{(n)}_T}J_{11}=0. 
\end{equation}

Combining \eqref{eq:74}, \eqref{eq:76} and \eqref{eq:77}, we eventually obtain
\[
\begin{aligned}
     \lim_{\delta \rightarrow 0}& \limsup_{n \to \infty} \sup_{\tau\in \Upsilon^{(n)}_T} \mathbb{P}(|B^{(n)}_{\tau + \delta} - B^{(n)}_{\tau}| > \varepsilon)\\
     &\leq \frac{1}{\varepsilon} \lim_{\delta \rightarrow 0} \limsup_{n \to \infty} \sup_{\tau\in \Upsilon^{(n)}_T}\bE\left[\left|f\left(B^{(n)}_{\tau + \delta}\right) - f\left(B^{(n)}_{\tau}\right)\right|\right]=0.
 \end{aligned}\]
 This establishes \eqref{eq:42}.  
\end{proof}

\subsection{Proof of weak convergence}

The proof of weak convergence follows the argument of Section~\ref{SEC5}. In completing the proof of the corresponding assertion in Proposition~\ref{PRO52}, it suffices to replace \eqref{eq:512} by the following fact:
 \begin{equation}\label{eq:78}
\lim_{n\rightarrow \infty}\bE\left[\sup_{t\in [0,T]}\left|M^{(n),2}_t\right|\right]=0.
\end{equation}
(Repeating the final step of the proof yields that the sequence of random variables in \eqref{eq:514} converges to $0$ in probability; this is what is required to apply \cite[Corollary 3.3 of Chapter 3]{B99}.) We now proceed to prove this statement.

\begin{proof}[Proof of \eqref{eq:78}]
    According to \eqref{eq:59}, it suffices to show that
    \[
        \lim_{n\rightarrow \infty}\bE\left[\int_0^T \left| \left(\sL f-\sL^{(n)}f\right)(B^{(n)}(s))\right|\dd s\right]=0.
    \]
    From \eqref{eq:510}, we know that
    \begin{equation}\label{eq:79-2}
     \bE\left[\int_0^T \left| \left(\sL f-\sL^{(n)}f\right)(B^{(n)}(s))1_{\{B^{(n)}_s\neq 0\}}\right|\dd s\right] \leq \frac{\delta_{f''}(1/n)}{2}\cdot T\rightarrow 0
    \end{equation}
    as $n\rightarrow \infty$. 
    In analogy with  \eqref{eq:515}, a direct computation yields 
    \[
        \sL^{(n)}f(0)=n\left(\frac{p_1}{p_2}(-f(0))+\fp^{(n)}_1\cdot f'(\theta^{(n)}_3)+\frac{\int_{(0,\infty)}h_n(x)\,p_4(\dd x)}{p_2}\right):=n\cdot \gamma(n),
    \]
    for some $\theta^{(n)}_3\in (0,1/n)$, where $h_n$ is given below \eqref{eq:515}. By \eqref{eq:319-2} and the boundary condition satisfied by $f$ in \eqref{eq:37},  the sequence $\gamma(n)$ satisfies
    \begin{equation}\label{eq:79}
        \lim_{n\rightarrow \infty}\gamma(n)=0. 
    \end{equation}
   Consequently, 
    \[
    \begin{aligned}
     &\bE\left[\int_0^T \left| \left(\sL f-\sL^{(n)}f\right)(B^{(n)}(s))1_{\{B^{(n)}_s= 0\}}\right|\dd s\right]  \\
     &\quad \leq |\sL f(0)-n\gamma(n)|\cdot \frac{1}{n^2} \bE\left[\sum_{k=0}^{\lfloor n^2T\rfloor+1}1_{\{X^{(n)}_k=0\}} \right].
    \end{aligned}\]
By utilizing \eqref{eq:71} (with $m:=\lfloor n^2T\rfloor+2$) and \eqref{eq:73}, we can further use \eqref{eq:79} to obtain
\[
   \lim_{n\rightarrow \infty} \bE\left[\int_0^T \left| \left(\sL f-\sL^{(n)}f\right)(B^{(n)}(s))1_{\{B^{(n)}_s= 0\}}\right|\dd s\right]=0. 
\]
This, together with \eqref{eq:79-2}, establishes \eqref{eq:78}. 
\end{proof}

After establishing the corresponding conclusion of Proposition~\ref{PRO52} for the present case $p_2\neq 0, p_3=0$, it suffices to replace \eqref{eq:517} by the estimate
\[
  \frac{1}{n^2}|\sL^{(n)}f(0)|\leq \frac{\gamma(n)}{n}. 
\] 
With this substitution, the proof of Theorem~\ref{THM35} can be completed.

\appendix

\section{Proof of  (\ref{eq:24-2})}\label{APPA}

In the following proof, we use $\bP_i$ (with corresponding expection $\bE_i$) to denote the law of the BRW with transition matrix \eqref{eq:22} starting from $i\in \bN$. The expectation $\bE$ appearing in the statement of Lemma~\ref{LM22} is understood as $\bE_0$. 

\begin{proof}[Proof of \eqref{eq:24-2}]
For convenience, set $a_k:=\bP(X_k=0)$ for $k\geq 0$. Clearly, $a_0=1$ and $a_1=\fp_0$. Let $\tau:=\inf\{k\geq 0: X_k=0\}$ denote the first return time of $X$ to $0$. 

Fix $k\geq 2$. By the Markov property of $X$, we have
\[
    a_k=\sum_{i=0}^\infty\bP(X_1=i)\bP_i(X_{k-1}=0)=\fp_0a_{k-1}+\sum_{i=1}^{k-1}\fp_i\cdot \bP_i(X_{k-1}=0).
\]
In addition, for each $i\geq 1$,
\begin{equation}\label{eq:A2}
\begin{aligned}
    \bP_i(X_{k-1}=0)&=\bP_i(X_{k-1}=0,i\leq \tau\leq k-1)\\
    &=\sum_{j=i}^{k-1}\bP_i(\bP_{X_j}(X_{k-j-1}=0);\tau=j)\\
    &=\sum_{j=i}^{k-1}a_{k-j-1}\bP_i(\tau=j).
\end{aligned}\end{equation}
The quantity $\mathbb{P}_i(\tau=j)$ is precisely the probability that a simple random walk starting from $i\ge 1$ hits 0 for the first time at step $j$. This probability is nonzero only when $j-i$ is even, say $j-i=2j_0$, and a direct computation yields
\begin{equation}\label{eq:A1}
    \bP_i(\tau=j)=\bP_i(\tau=i+2j_0)=\frac{M_{i,j_0}}{2^{i+2j_0}},
\end{equation}
where 
\[
M_{i,j_0}=\binom{i+2j_0-1}{j_0}-\binom{i+2j_0-1}{j_0-1}= \frac{i}{i+2j_0} \binom{i+2j_0}{j_0}
\]
are the so-called \emph{shifted Catalan numbers}. Their generating function is given by
\begin{equation}\label{eq:A3}
M_i(t) := \sum_{j_0=0}^\infty M_{i,j_0} t^{j_0} = \left( \frac{1 - \sqrt{1 - 4t}}{2t} \right)^i,\quad 0<t\leq 1/4;
\end{equation}
see, e.g.,  \cite[\S1.3, (1.3)]{S15} and  \cite[\S5.4, (5.60)]{GKP91}. Substituting \eqref{eq:A1} into \eqref{eq:A2} yields 
\[
    \bP_i(X_{k-1}=0)=\sum_{j_0=0}^{\lfloor\frac{k-i-1}{2} \rfloor}a_{k-i-2j_0-1}\frac{M_{i,j_0}}{2^{i+2j_0}}. 
\]
Then it follows that
\[
    a_k=\fp_0a_{k-1}+\sum_{i=1}^{k-1}\fp_i\sum_{j_0=0}^{\lfloor\frac{k-i-1}{2} \rfloor}a_{k-i-2j_0-1}\frac{M_{i,j_0}}{2^{i+2j_0}},\quad k\geq 2. 
\]

We now turn to the computation of $F(x)$. 
For $0\leq x<1$, we have
\[
\begin{aligned}
    F(x)&=1+\fp_0x+\sum_{k= 2}^\infty \left[\fp_0a_{k-1}+\sum_{i=1}^{k-1}\fp_i\sum_{j_0=0}^{\lfloor\frac{k-i-1}{2} \rfloor}a_{k-i-2j_0-1}\frac{M_{i,j_0}}{2^{i+2j_0}}\right]\,x^k \\
    &=1+\fp_0 xF(x)+\sum_{k= 2}^\infty\sum_{i=1}^{k-1}\fp_i\left(\frac{x}{2}\right)^{i}\left[\sum_{j_0=0}^{\lfloor\frac{k-i-1}{2} \rfloor}a_{k-i-2j_0-1}x^{k-i-2j_0}M_{i,j_0}\left(\frac{x^2}{4}\right)^{j_0}\right].
\end{aligned}\]
By first interchanging the order of summation in $k$ and $i$, and then making the change of variables $k':=k-i$, we obtain
\[
    F(x)-1-\fp_0 xF(x)=\sum_{i=1}^\infty \fp_i\left(\frac{x}{2}\right)^{i}\left[\sum_{k'=1}^\infty\sum_{j_0=0}^{\lfloor\frac{k'-1}{2} \rfloor}a_{k'-2j_0-1}x^{k'-2j_0}M_{i,j_0}\left(\frac{x^2}{4}\right)^{j_0}\right].
\]
For each fixed $i\geq 1$, we may further interchange the order of summation in $k'$ and $j_0$, and then set $k'':=k'-2j_0$, which yields
\[
\begin{aligned}
    \sum_{k'=1}^\infty&\sum_{j_0=0}^{\lfloor\frac{k'-1}{2} \rfloor}a_{k'-2j_0-1}x^{k'-2j_0}M_{i,j_0}\left(\frac{x^2}{4}\right)^{j_0} \\
    &=\sum_{j_0=0}^\infty M_{i,j_0}\left(\frac{x^2}{4}\right)^{j_0}\sum_{k''=1}^\infty a_{k''-1}x^{k''}=M_i(x^2/4)\cdot xF(x),
\end{aligned}\]
where $M_i$ is the generating function given by \eqref{eq:A3}. As a result,
\[
F(x)-1-\fp_0 xF(x)=xF(x)\sum_{i=1}^\infty \fp_i\left(\frac{x}{2}\right)^{i}M_i(x^2/4)=xF(x)\sum_{i=1}^\infty \fp_i \left( \frac{1-\sqrt{1-x^2}}{x}\right)^i.
\]
This identity immediately leads to \eqref{eq:24-2}, completing the proof. 
\end{proof}

\bibliographystyle{siam} % We choose the "plain" reference style
\bibliography{scalinglimit} % Entries are in the "refs.bib" file

@book{RY99,
author = {Revuz, Daniel and Yor, Marc},
title = {{Continuous martingales and {B}rownian motion}},
publisher = {Springer-Verlag, Berlin},
year = {1999},
volume = {293},
series = {Grundlehren der Mathematischen Wissenschaften},
address = {Berlin, Heidelberg},
edition = {3rd}
}

@book{F99,
author = {Folland, Gerald B},
title = {{Real analysis}},
publisher = {John Wiley {\&} Sons, Inc., New York},
year = {1999},
series = {Pure and Applied Mathematics (New York)},
edition = {Second}
}

@article{IM63, title={Brownian motions on a half line}, volume={7}, DOI={10.1215/ijm/1255644633}, number={2}, fjournal={Illinois Journal of Mathematics},
  journal={Ill. J. Math.},
  author={It\^o, K. and McKean, H. P.}, year={1963}, pages={181–231}, language={en} }

@book{IM65, title={Diffusion processes and their sample paths}, publisher={Springer-Verlag, Berlin-New York}, author={It\^o, Kiyosi and McKean Jr, Henry P.}, year={1965} }

@article{F54, title={Diffusion processes in one dimension}, volume={77}, DOI={10.1090/S0002-9947-1954-0063607-6}, number={1}, fjournal={Transactions of the American Mathematical Society},
  journal={Trans. Am. Math. Soc.},
  author={Feller, William}, year={1954}, pages={1–31}, language={English} }

@article{BC08, title={Discrete approximations to reflected {B}rownian motion}, volume={36}, DOI={10.1214/009117907000000240}, number={2}, fjournal={The Annals of Probability},
 journal={Ann. Probab.},
 author={Burdzy, Krzysztof and Chen, Zhen-Qing}, year={2008}, pages={698–727}, language={en-US} }

@article{D51, series={American mathematical society. Memoirs}, title={An invariance principle for certain probability limit theorems}, url={https://books.google.de/books?id=8ljTMgAACAAJ}, author={Donsker, M.D.}, year={1951}, journal={Memoirs American Mathematical Society}, language={en} }

@article{D52, title={Justification and extension of {D}oob's heuristic approach to the {K}olmogorov- {S}mirnov theorems}, volume={23}, ISSN={0003-4851}, DOI={10.1214/aoms/1177729445}, number={2}, fjournal={The Annals of Mathematical Statistics},
journal={Ann. Math. Stat.},
author={Donsker, Monroe D.}, year={1952}, month=jun, pages={277–281}, language={en} }

@article{HL81, title={Sticky {B}rownian motion as the limit of storage processes}, volume={18}, rights={https://www.cambridge.org/core/terms}, ISSN={0021-9002, 1475-6072}, DOI={10.2307/3213181},  number={1}, fjournal={Journal of Applied Probability},
 journal={J. Appl. Probab.},
 author={Harrison, J. Michael and Lemoine, Austin J.}, year={1981}, month=mar, pages={216–226}, language={en} }

@article{A91, title={Sticky {B}rownian motion as the strong limit of a sequence of random walks}, volume={39}, rights={https://www.elsevier.com/tdm/userlicense/1.0/}, ISSN={03044149}, DOI={10.1016/0304-4149(91)90080-V}, number={2}, fjournal={Stochastic Processes and their Applications},
 journal={Stochastic Processes Appl.},
 author={Amir, Madjid}, year={1991}, month=dec, pages={221–237}, language={en} }

@article{EFJP24, title={A functional central limit theorem for the general {B}rownian motion on the half-line}, url={http://arxiv.org/abs/2408.06830}, DOI={10.48550/arXiv.2408.06830}, abstractNote={In this work, we establish a Trotter-Kato type theorem. More precisely, we characterize the convergence in distribution of Feller processes by examining the convergence of their generators. The main novelty lies in providing quantitative estimates in the vague topology at any fixed time. As important applications, we deduce functional central limit theorems for random walks on the positive integers with boundary conditions, which converge to Brownian motions on the positive half-line with boundary conditions at zero.}, journal={arXiv:2408.06830}, number={arXiv:2408.06830}, publisher={arXiv}, author={Erhard, Dirk and Franco, Tertuliano and Jara, Milton and Pimenta, Eduardo}, year={2024}, month=nov, language={en} }

@article{ABP25, title={Scaling limit of boundary random walks: A martingale problem approach}, url={http://arxiv.org/abs/2507.10528}, DOI={10.48550/arXiv.2507.10528}, abstractNote={We establish the scaling limit of a class of boundary random walks to Brownian-type processes on the half-line. By solving the associated martingale problem and employing weak convergence techniques, we prove that under appropriate scaling, the process converges to the general Brownian motion in the $J_1$-Skorokhod topology. The limiting process exhibits both diffusion and boundary behavior characterized by parameters $(alpha, beta, A, B)$, which govern the transition rates at the origin. Our results provide a discrete approximation to generalized Brownian motions with mixed boundary conditions.}, journal={arXiv:2507.10528}, number={arXiv:2507.10528}, publisher={arXiv}, author={Arroyave, Juan Carlos and Barros, Eldon and Pimenta, Eduardo}, year={2025}, month=jul}

@article{Al78, title={Stopping times and tightness}, volume={6}, ISSN={0091-1798, 2168-894X}, DOI={10.1214/aop/1176995579}, abstractNote={A sufficient condition for the tightness of a sequence of stochastic processes is given in terms of their behavior after stopping times. As an application, the conditions for McLeish’s invariance principle for martingales are weakened.}, number={2}, journal={Ann. Probab.}, fjournal={Annals of Probability}, publisher={Institute of Mathematical Statistics}, author={Aldous, David}, year={1978}, month=apr, pages={335–340}, language={en} }

@book{JS03, address={Berlin Heidelberg New York}, edition={2nd}, series={Grundlehren der mathematischen Wissenschaften}, title={Limit theorems for stochastic processes}, ISBN={978-3-540-43932-5}, callNumber={519.236}, publisher={Springer}, author={Jacod, Jean and Širâev, Alʹbert Nikolaevič}, year={2003}, collection={Grundlehren der mathematischen Wissenschaften}, language={en} }

@book{Lig10, address={Providence, R.I}, series={Graduate studies in mathematics}, title={Continuous time {M}arkov processes: An introduction}, ISBN={978-0-8218-4949-1}, callNumber={519.233}, publisher={American Mathematical Society}, author={Liggett, Thomas M.}, year={2010}, collection={Graduate studies in mathematics}, language={en} }

@book{EK86, edition={1}, series={Wiley Series in Probability and Statistics}, title={Markov processes: Characterization and convergence}, rights={http://doi.wiley.com/10.1002/tdm_license_1.1}, ISBN={978-0-471-08186-9}, url={https://onlinelibrary.wiley.com/doi/book/10.1002/9780470316658}, DOI={10.1002/9780470316658}, publisher={Wiley}, author={Ethier, Stewart N. and Kurtz, Thomas G.}, year={1986}, month=mar, collection={Wiley Series in Probability and Statistics}, language={en} }

@article{Li25, title={Birth-death processes are time-changed {F}eller’s {B}rownian motions}, volume={190}, ISSN={03044149}, DOI={10.1016/j.spa.2025.104738}, journal={Stochastic Processes Appl.}, fjournal={Stochastic Processes and their Applications}, author={Li, Liping}, year={2025}, month=dec, pages={104738}, language={en} }

@article{F52, title={The parabolic differential equations and the associated semi-groups of transformations}, volume={55}, ISSN={0003-486X}, DOI={10.2307/1969644}, number={3}, journal={Ann. Math.}, fjournal={Annals of Mathematics}, publisher={[Annals of Mathematics, Trustees of Princeton University on Behalf of the Annals of Mathematics, Mathematics Department, Princeton University]}, author={Feller, William}, year={1952}, pages={468–519}, language={en} }

@book{S15, address={New York, NY}, edition={1st}, title={Catalan numbers}, ISBN={978-1-107-07509-2}, publisher={Cambridge Univ. Press}, author={Stanley, Richard P.}, year={2015}, language={en} }

@book{GKP91, address={Reading (Mass.) Menlo Park (Calif.) Paris}, title={Concrete mathematics: A foundation for computer science}, ISBN={978-0-201-14236-5}, callNumber={004.015 1}, publisher={Addison-Wesley}, author={Graham, Ronald Lewis and Knuth, Donald Ervin and Patashnik, Oren}, year={1991}, language={en} }

@book{B99, address={New York Weinheim}, edition={2. ed}, series={Wiley series in probability and statistics Probability and statistics section}, title={Convergence of probability measures}, ISBN={978-0-471-19745-4}, publisher={Wiley}, author={Billingsley, Patrick}, year={1999}, collection={Wiley series in probability and statistics Probability and statistics section}, language={en} }

\end{document}